\documentclass[reqno]{amsart}
\usepackage[left=1in, right=1in, top=1in]{geometry}

\usepackage{etoolbox}

\makeatletter
\let\old@tocline\@tocline
\let\section@tocline\@tocline
\newcommand{\section@dotsep}{4.5}
\newcommand{\subsection@dotsep}{4.5}
\patchcmd{\@tocline}
  {\hfil}
  {\nobreak
     \leaders\hbox{$\m@th
        \mkern \section@dotsep mu\hbox{.}\mkern \section@dotsep mu$}\hfill
     \nobreak}{}{}
\let\section@tocline\@tocline
\let\@tocline\old@tocline

\patchcmd{\@tocline}
  {\hfil}
  {\nobreak
     \leaders\hbox{$\m@th
        \mkern \subsection@dotsep mu\hbox{.}\mkern \subsection@dotsep mu$}\hfill
     \nobreak}{}{}
\let\subsection@tocline\@tocline
\let\@tocline\old@tocline

\let\old@l@section\l@section
\let\old@l@subsection\l@subsection

\def\@tocwriteb#1#2#3{%
  \begingroup
    \@xp\def\csname #2@tocline\endcsname##1##2##3##4##5##6{%
      \ifnum##1>\c@tocdepth
      \else \sbox\z@{##5\let\indentlabel\@tochangmeasure##6}\fi}%
    \csname l@#2\endcsname{#1{\csname#2name\endcsname}{\@secnumber}{}}%
  \endgroup
  \addcontentsline{toc}{#2}%
    {\protect#1{\csname#2name\endcsname}{\@secnumber}{#3}}}%

\newlength{\@tocsectionindent}
\newlength{\@tocsubsectionindent}
\newlength{\@tocsubsubsectionindent}
\newlength{\@tocsectionnumwidth}
\newlength{\@tocsubsectionnumwidth}
\newlength{\@tocsubsubsectionnumwidth}
\newcommand{\settocsectionnumwidth}[1]{\setlength{\@tocsectionnumwidth}{#1}}
\newcommand{\settocsubsectionnumwidth}[1]{\setlength{\@tocsubsectionnumwidth}{#1}}
\newcommand{\settocsubsubsectionnumwidth}[1]{\setlength{\@tocsubsubsectionnumwidth}{#1}}
\newcommand{\settocsectionindent}[1]{\setlength{\@tocsectionindent}{#1}}
\newcommand{\settocsubsectionindent}[1]{\setlength{\@tocsubsectionindent}{#1}}
\newcommand{\settocsubsubsectionindent}[1]{\setlength{\@tocsubsubsectionindent}{#1}}

\renewcommand{\l@section}{\section@tocline{1}{\@tocsectionvskip}{\@tocsectionindent}{}{\@tocsectionformat}}%
\renewcommand{\l@subsection}{\subsection@tocline{1}{\@tocsubsectionvskip}{\@tocsubsectionindent}{}{\@tocsubsectionformat}}%
\renewcommand{\l@subsubsection}{\subsubsection@tocline{1}{\@tocsubsubsectionvskip}{\@tocsubsubsectionindent}{}{\@tocsubsubsectionformat}}%
\newcommand{\@tocsectionformat}{}
\newcommand{\@tocsubsectionformat}{}
\newcommand{\@tocsubsubsectionformat}{}
\expandafter\def\csname toc@1format\endcsname{\@tocsectionformat}
\expandafter\def\csname toc@2format\endcsname{\@tocsubsectionformat}
\expandafter\def\csname toc@3format\endcsname{\@tocsubsubsectionformat}
\newcommand{\settocsectionformat}[1]{\renewcommand{\@tocsectionformat}{#1}}
\newcommand{\settocsubsectionformat}[1]{\renewcommand{\@tocsubsectionformat}{#1}}
\newcommand{\settocsubsubsectionformat}[1]{\renewcommand{\@tocsubsubsectionformat}{#1}}
\newlength{\@tocsectionvskip}
\newcommand{\settocsectionvskip}[1]{\setlength{\@tocsectionvskip}{#1}}
\newlength{\@tocsubsectionvskip}
\newcommand{\settocsubsectionvskip}[1]{\setlength{\@tocsubsectionvskip}{#1}}
\newlength{\@tocsubsubsectionvskip}
\newcommand{\settocsubsubsectionvskip}[1]{\setlength{\@tocsubsubsectionvskip}{#1}}

\patchcmd{\tocsection}{\indentlabel}{\makebox[\@tocsectionnumwidth][l]}{}{}
\patchcmd{\tocsubsection}{\indentlabel}{\makebox[\@tocsubsectionnumwidth][l]}{}{}
\patchcmd{\tocsubsubsection}{\indentlabel}{\makebox[\@tocsubsubsectionnumwidth][l]}{}{}

\newcommand{\@sectypepnumformat}{}
\renewcommand{\contentsline}[1]{%
  \expandafter\let\expandafter\@sectypepnumformat\csname @toc#1pnumformat\endcsname%
  \csname l@#1\endcsname}
\newcommand{\@tocsectionpnumformat}{}
\newcommand{\@tocsubsectionpnumformat}{}
\newcommand{\@tocsubsubsectionpnumformat}{}
\newcommand{\setsectionpnumformat}[1]{\renewcommand{\@tocsectionpnumformat}{#1}}
\newcommand{\setsubsectionpnumformat}[1]{\renewcommand{\@tocsubsectionpnumformat}{#1}}
\newcommand{\setsubsubsectionpnumformat}[1]{\renewcommand{\@tocsubsubsectionpnumformat}{#1}}
\renewcommand{\@tocpagenum}[1]{%
  \hfill {\mdseries\@sectypepnumformat #1}}
  
\let\oldappendix\appendix
\renewcommand{\appendix}{%
  \leavevmode\oldappendix%
  \addtocontents{toc}{%
    \protect\settowidth{\protect\@tocsectionnumwidth}{\protect\@tocsectionformat\sectionname\space}%
    \protect\addtolength{\protect\@tocsectionnumwidth}{2em}}%
}
\makeatother

\makeatletter
\settocsectionnumwidth{2em}
\settocsubsectionnumwidth{2.5em}
\settocsubsubsectionnumwidth{3em}
\settocsectionindent{1pc}%
\settocsubsectionindent{\dimexpr\@tocsectionindent+\@tocsectionnumwidth}%
\settocsubsubsectionindent{\dimexpr\@tocsubsectionindent+\@tocsubsectionnumwidth}%
\makeatother

\settocsectionvskip{0pt}
\settocsubsectionvskip{0pt}
\settocsubsubsectionvskip{0pt}

\usepackage{mathrsfs}
\usepackage{mathrsfs}
\usepackage{amsfonts}
\usepackage{comment}
\usepackage{cases}
\usepackage{latexsym}
\usepackage{amsmath}
\usepackage[all]{xy}
\usepackage{stmaryrd}
\usepackage{amsfonts}
\usepackage{amsmath,amssymb,amscd,bbm,amsthm,mathrsfs,dsfont}

\usepackage{tikz}

\usepackage{pgflibraryarrows}
\usepackage{pgflibrarysnakes}

\usepackage[numbers,sort&compress]{natbib}
\usepackage{hypernat}
 
\usepackage{color, hyperref}
\definecolor{blue}{rgb}{0.9,0.0,0.9}
\hypersetup{colorlinks, breaklinks,
            linkcolor=red,urlcolor=blue,
            anchorcolor=blue,citecolor=blue}

\usepackage{fancyhdr}
\usepackage{amsxtra,ifthen}
\usepackage{verbatim}

\usepackage{fancyhdr}
\usepackage{amsxtra,ifthen}
\usepackage{verbatim}

\numberwithin{equation}{section}

\theoremstyle{plain}
\newtheorem{theorem}{Theorem}[section]
\newtheorem{lemma}[theorem]{Lemma}
\newtheorem{proposition}[theorem]{Proposition}

\newtheorem{corollary}[theorem]{Corollary}

\theoremstyle{definition}
\newtheorem{definition}[theorem]{Definition}

\newtheorem{remark}[theorem]{Remark}
\newtheorem{question}[theorem]{Question}

\allowdisplaybreaks[4]

\begin{document}

\title[Nonlinear Lie-Type Derivations of Incidence Algebras]
{Nonlinear Lie-Type Derivations of finitary Incidence Algebras and Related Topics}

\author{Yuping Yang}
\address{Yang: School of Mathematics and statistics, Southwest University, Chongqing 400715, P. R. China}
\email{yupingyang@swu.edu.cn}


\author{Feng Wei*}\thanks{*Corresponding author}
\address{Wei: School of Mathematics and Statistics, Beijing
Institute of Technology, Beijing, 100081, China}
\email{daoshuo@hotmail.com} \email{daoshuowei@gmail.com}

\begin{abstract}
This is a continuation of our earlier works \cite{KhrypchenkoWei, Yang20211, Yang20212} 
with respect to (non-)linear Lie-type derivations of finitary incidence algebras. 
Let $X$ be a pre-ordered set, $\mathcal{R}$ be a $2$-torsionfree and $(n-1)$-torsionfree 
commutative ring with identity, where $n\geq 2$ is an integer.
Let $FI(X,\mathcal{R})$ be the finitary incidence algebra of $X$ over $\mathcal{R}$.  In this paper, 
a complete clarification is obtained for the structure of nonlinear Lie-type derivations of $FI(X,\mathcal{R})$.
 We introduce a new class of derivations on $FI(X,\mathcal{R})$ named inner-like derivations, and prove 
 that each nonlinear Lie $n$-derivation on $FI(X,\mathcal{R})$ is the sum of an inner-like derivation, a transitive induced derivation and 
a quasi-additive induced Lie $n$-derivation. Furthermore, if $X$ is finite, we show that a 
quasi-additive induced Lie $n$-derivation can be expressed as the sum of 
an additive induced Lie derivation and a central-valued map annihilating all $(n-1)$-th commutators. 
We also provide a sufficient and necessary condition such that every nonlinear Lie $n$-derivation of $FI(X,\mathcal{R})$ is of proper form. 
Some related topics for further research are proposed in the last section of this article.
 \end{abstract}



\thanks{The work of the first author is supported by National Natural Science Foundation of China (No. 11701468).}

\date{\today}

\maketitle

\tableofcontents

\setcounter{section}{-1}

\section{Introduction}\label{xxsec0}

Let $\mathcal{A}$ be an associative algebra over a commutative ring $\mathcal{R}$. 
Let $[\ ,\ ]$ and $\circ$ be the {\it Lie product} and {\it Jordan product} respectively, i.e., $[x,y]=xy-yx$ and $x\circ y=xy+yx$ for all $x,y\in \mathcal{A}$. 
Then $(\mathcal{A},[\ ,\ ])$  is a Lie algebra and $(\mathcal{A},\circ)$ is a Jordan algebra. It is a fascinating topic to study the 
connection between the associative, Lie and Jordan structures of $\mathcal{A}$. In this field, there are two classes of important algebraic 
maps: algebra homomorphisms and differential operators. For example Jordan homomorphisms, Lie homomorphisms,
 Jordan derivations and Lie derivations. In the AMS Hour Talk of 1961, Herstein proposed many problems concerning the structures of 
Jordan and Lie derivations in associative simple and prime rings~\cite{Her}. Roughly speaking, 
he conjectured that these derivations are all of the proper or standard form. The renowned 
Herstein's Lie-type mapping research program was formulated since then. Martindale 
gave a driving force in this program under the assumption that the rings contain some 
nontrivial idempotents \cite{Mar64, Mar69}. The first idempotent-free result 
on Lie-type maps was obtained by Bre\v sar in \cite{Bre93}.

Recall that  an $\mathcal{R}$-linear map
$D: \mathcal{A}\longrightarrow \mathcal{A}$ is called a \textit{derivation} if $D(xy)=D(x)y+xD(y)$ for all $x,y\in \mathcal{A}$, and
an $\mathcal{R}$-linear map $L: \mathcal{A}\longrightarrow \mathcal{A}$ is called a \textit{Lie derivation} if
$$
L([x,y])=[L(x),y]+[x,L(y)]
$$
for all $x,y\in \mathcal{A}$.  Obviously, a derivation is a Lie derivation. But the converse 
statements are not true in general. For instance, suppose that $d: \mathcal{A}\longrightarrow 
A$ is a derivation and that $\tau: \mathcal{A}\longrightarrow \mathcal{C(A)}$ is an $\mathcal{R}$-linear 
map from $\mathcal{A}$ into its center $\mathcal{C(A)}$ such that $\tau([x,y])=0$ for all $x, y\in \mathcal{A}$. 
Then $d+\tau$ is a Lie derivation of $\mathcal{A}$, but it is not necessarily a derivation
of $\mathcal{A}$.  For each positive integer $n$, we define inductively a polynomial $p_n(x_1,x_2,\cdots,x_n)$ as follows:
$$
\begin{aligned}
p_1(x_1)&=x_1,\\
p_2(x_1,x_2)&=[x_1,x_2]=x_1x_2-x_2x_1,\\
p_3(x_1,x_2,x_3)&=[p_2(x_1,x_2),x_3]=[[x_1,x_2],x_3],\\
\vdots \ \ \ \ \ \ & \ \ \ \ \ \ \ \ \ \ \ \ \ \vdots \\
p_n(x_1,x_2,\cdots,x_n)&=[p_{n-1}(x_1,x_2,\cdots,x_{n-1}),x_n]\\
&=\underbrace{[[\cdots [[}_{n-1}x_1,  x_2], x_3],  \cdots,  x_{n-1}],  x_n].
\end{aligned}\eqno{(\clubsuit)}
$$
When $x_1,\cdots,x_n\in \mathcal{A}$, $n\geq 2$, then element $p_n(x_1,x_2,\cdots,x_n)$ is said to be an
$(n-1)$-\textit{th commutator} of $\mathcal{A}$. An $\mathcal{R}$-linear map $L:
\mathcal{A}\longrightarrow  \mathcal{A}$ is called a
\textit{Lie $n$-derivation} if
$$ 
L(p_n(x_1,x_2,\dots,x_n))=\sum_{k=1}^n
p_n(x_1,\dots,x_{k-1}, L(x_k),x_{k+1},\dots,x_n) \eqno{(\spadesuit)} 
$$
holds for all $x_1,x_2,\dots,x_n\in \mathcal{A}$. Lie $n$-derivations with $n\geq 2$ are called \textit{lie-type derivations}.

Lie $n$-derivation was introduced by Abdullaev
\cite{Abdullaev}. In particular, he showed that every Lie $n$-derivation $L$ on a von Neumann algebra $\mathcal{M}$ 
without central summands of type $\mathcal{I}_1$ can be decomposed as $L=D+\kappa$, where $D$ is an ordinary 
derivation on $\mathcal{M}$ and $\kappa$ is a center-valued map
annihilating all $(n-1)$-th commutators.  A Lie $n$-derivation of $\mathcal{A}$ is called \textit{proper} 
if it is the sum of a derivation and a central-valued linear map annihilating all $(n-1)$-th commutator of $\mathcal{A}$.
The question of whether each Lie-type derivation on a given algebra has the proper form  is extensively 
studied, see \cite{Bre93, Mar64, Qi1, Qi2, Wang, YuZhang2}. 
These works totally fulfill the Herstein's program in the background of noncommutative algebras and operator algebras.  

For the afore-mentioned definitions, if the $\mathcal{R}$-linear assumption is removed, then the 
corresponding derivation and Lie $n$-derivation become \textit{nonlinear derivation} and  \textit{nonlinear Lie $n$-derivation} respectively.
All the nonlinear Lie $n$-derivations with $n\geq 2$ are considered as \textit{nonlinear Lie-type derivations}. Thus far, 
plenty of interesting results have been obtained in this vein. In \cite{ChenZhang}, Chen and Zhang studied nonlinear Lie derivations of triangular algebras. They proved that every nonlinear 
Lie derivation of a triangular algebra is the sum of an additive derivation and a central-valued map annihilating all the
commutators.  Ji et al \cite{JiLiuZhao} obtained an analogous result for nonlinear Lie triple derivations 
of triangular algebras.  Benkovi\v c and Eremita \cite{BenkovicEremita} showed that every nonlinear Lie $n$-derivation 
on a class of triangular algebras can be expressed as the sum of an additive derivation and a central-valued map annihilating all $(n-1)$-th commutators. This result was 
extended to the case of generalized matrix algebras by Wang ad Wang \cite{WangWang}. Fosner, 
Wei and Xiao \cite{FosnerWeiXiao,XiaoWei} studied nonlinear lie-type derivations on full matrix algebras and von Neumann algebras.
 As a summary, the previous works showed that nonlinear lie-type derivations of triangular algebras, generalized 
 matrix algebras and some operator algebras are of proper forms.
This inspired us to consider what we can say about nonlinear Lie-type derivations of finitary incidence algebras.

Now let us introduce some basic facts about finitary incidence algebras.
Let $(X,\leqslant)$ be a pre-ordered set. Denote by $I(X,\mathcal{R})$ be the set of functions
$$
I(X,\mathcal{R}):=\{f: X\times X\longrightarrow \mathcal{R}\mid f(x,y)=0\ \text{if}\ x\nleqslant y\}.
$$
It is clear that $I(X,\mathcal{R})$ is a $\mathcal{R}$-module, which is called the \textit{incidence space} 
of $X$ over $\mathcal{R}$.  A function $f\in  I(X,\mathcal{R})$ is said to be \textit{finitary} 
if for all $u<v$, there are finite number $u\leq x<y\leq v$ such that 
$f(x,y)\neq 0$. Let $FI(X,\mathcal{R})\subset I(X,\mathcal{R})$ be the set of finitary functions. 
Then $FI(X,\mathcal{R})$ is an algebra which is called {\it finitary incidence algebra} of $X$ over $\mathcal{R}$, with multiplication  given by the convolution 
$$
\begin{aligned}
(fg)(x,y)&=\sum_{x\leqslant z\leqslant y}f(x,z)g(z,y).
\end{aligned}
$$
for all $f,g\in I(X,\mathcal{R})$ and $x,y\in X$. When $X$ is locally finite, i.e., for all $x\leq y$, 
there are finite many $z\in X$ such that $x\leq z\leq y$, 
then $FI(X,\mathcal{R})=I(X,\mathcal{R})$ are called \textit{incidence algebra} of $X$ over $\mathcal{R}$.
It is also helpful to point out that the full matrix algebra $\mathcal{{}M}_n(\mathcal{R})$,
the upper (or lower) triangular matrix algebras $\mathcal{T}_n(\mathcal{R})$, and the infinite
triangular matrix algebras $\mathcal{T}_{\infty}(\mathcal{R})$ are examples of finitary incidence algebras. 
In the theory of operator algebras, the incidence algebra of a finite poset is 
referred as a digraph algebra or a finite dimensional CSL algebra. 

The incidence algebra of a poset was first considered by 
Ward in \cite{Ward} as a generalized algebra of arithmetic functions. Rota and Stanley 
developed incidence algebras as fundamental structures of enumerative combinatorial
theory and the allied areas of arithmetic function theory (see~\cite{Stanley2}). Furthermore, 
Stanley~\cite{Stanley1} initiated the study of algebraic maps and combinatorial structure of an 
incidence algebra. Since then, automorphisms, involutions, Lie derivation and Jordan derivations of finitary incidence 
algebras and related topics have been increasingly significant, see 
\cite{Baclawski, BruFK1, BruFK2, BruFS3, BruL, CourtemancheDugasHerden, Dugas, DugasWagner, 
DugasHerdenRebrovich1, DugasHerdenRebrovich2, KKW, Kh-aut, Kh-der, Kh-Jor, Kh-loc, 
KhrypchenkoWei, Kopp, Spiegel1, Spiegel2, SpiegelDonnell, WangXiao, Xiao, XiaoYang, ZhangKhrypchenko} 
and the references therein. Recently, the nonlinear derivations and nonlinear Lie derivations of incidence algebras 
were studied in \cite{Yang2020, Yang20211, Yang20212}. A new phenomenon shows that for general incidence algebras, 
there exist nonlinear Lie derivations which are not proper. In \cite{KhrypchenkoWei}, Khrypchenko and Wei proposed the following question.

\begin{question}\cite[Problem 2.15]{KhrypchenkoWei}\label{xxsec0.1}
Let $X$ be a pre-ordered set, $FI(X,\mathcal{R})$ be the finitary incidence algebra of $X$ over $\mathcal{R}$. Let $L\colon
FI(X,\mathcal{R})\longrightarrow FI(X,\mathcal{R})$ be a nonlinear Lie $n$-derivation. Under what conditions, 
does $L$ has a proper form?
\end{question}

The main aim of the present paper is to study the structure of nonlinear lie-type derivations 
of finitary incidence algebras, and give an answer to Question \ref{xxsec0.1}.
This paper is organized as follows. In Section \ref{xxsec1}, we give some necessary definitions and 
basic facts on finitary incidence algebras and nonlinear Lie-type derivations. 
In Section \ref{xxsec2}, we introduce a new kind of derivations of finitary incidence algebras, 
which are called inner-like derivations. It is shown that for every nonlinear Lie $n$-derivation $L$ 
of $FI(X,\mathcal{R})$, there is an inner-like derivation associated to $L$. Section \ref{xxsec3} 
is devoted to the structure of nonlinear Lie-type derivations of $FI(X,\mathcal{R})$. We prove 
that each nonlinear Lie $n$-derivation of $FI(X,\mathcal{R})$ is 
the sum of an inner-like derivation, a transitive induced derivation and a quasi-additive induced Lie $n$-derivation. 
In Section \ref{xxsec4}, we give a sufficient and necessary 
condition such that every nonlinear Lie-type derivation on $FI(X,\mathcal{R})$ is of proper form. 
Before ending up this article, we introduce some related topics and propose some potential problems for further research 
in Section \ref{xxsec5}.

\section{Preliminaries}
\label{xxsec1}

In this section, we will give some necessary 
definitions and basic facts related to finitary incidence algebras and nonlinear Lie-type derivations.

\subsection{Finitary incidence algebras}
Throughout this paper, let $\mathcal{R}$ be a $2$-torsionfree
and $(n-1)$-torsionfree commutative ring with identity and $FI(X, \mathcal{R})$ be the finitary incidence algebra of 
a pre-ordered set $(X,\leqslant)$ over $\mathcal{R}$.  Here $n\geq 2$ is a fixed integer. 
For each pair $x \le y$ in $X$, let $e_{xy}: X\times X\longrightarrow \mathcal{R}$ be the function given by
\begin{equation*}
e_{xy}(u,v)=\left\{
\begin{aligned} 1, \quad & \text{if} \ (x, y)=(u, v); \\
0,\quad & \text{otherwise.} \end{aligned}\right.
\end{equation*}
It is not difficult to see that the product of $FI(X,\mathcal{R})$ satisfies the relation ${e_{xy}}{e_{uv}} = {\delta _{yu}}{e_{xv}}$,
and that each element $\alpha\in FI(X,\mathcal{R})$ is of the form 
\begin{equation}
\alpha=\sum_{x\leq y}\alpha(x,y)e_{xy}.
\end{equation}
Furthermore, $\alpha$ has the property: for all $u<v$ in $X$, there are finite number of $u\leq x<y\leq v$ with $\alpha(x,y)\neq 0$.
For convenience,  we give the following notations which will be invoked frequently throughout this paper:
\begin{itemize}
\item $x\sim y$ if $x\leq y$ or $y\leq x$;
\item  $x<y$ if $x\leq y$ and $x\neq y$;
\item $x\simeq y$ if $x\leq y$ and $y\leq x$;
\item $\mathcal{S}_{xy} = \{ \, r {e_{xy}}\, |\, r \in \mathcal{R} \} $ for $x<y$;
\item ${\bf E}=\{e_{xy}\mid x<y\}$.
\end{itemize}
For two elements $x,y\in X$, we say that $x$ and $y$ are {\it connected} if there exist 
$x_1,\cdots, x_n\in X$ such that $x\sim x_1, x_1\sim x_2,\cdots , x_n\sim y$. 
Let $X= {\bigsqcup\limits_{i \in \mathcal{J}}}{X_i}$ be the decomposition of $X$ into the union of its connected components, 
where $\mathcal{J}$ is the index set. For each $j\in \mathcal{J}$, we will say that $FI(X_j,\mathcal{R})$ 
is a \textit{connected component} of $FI(X,\mathcal{R})$. 
 Let $\mathcal{C}(X_i, \mathcal{R})$ and $\mathcal{C}(X, \mathcal{R})$ be the centers of $FI(X_i, \mathcal{R})$
and $FI(X, \mathcal{R})$, respectively. Let $I_i: = \sum\limits_{x \in {X_i}} {{e_{xx}}}$ for each $i \in \mathcal{J}$.
Then $\mathcal{C}(X_i, \mathcal{R})$ is $\mathcal{R}$-linearly spanned by $I_i$ and 
$\mathcal{C}(X, \mathcal{R})={\bigoplus\limits_{i \in \mathcal{J}}}{\mathcal{C}(X_i, \mathcal{R})}$. 
Let $\mathcal{D}(X, \mathcal{R})$ be the set of diagonal elements in $FI(X, \mathcal{R})$, i.e., 
\begin{align}
\mathcal{D}(X, \mathcal{R}) = \left\{ \, {\sum\limits_{x \in X} {{d_x}{e_{xx}}} \, \mid \, {d_x} \in \mathcal{R}, x \in X} \, \right\}.
\end{align}

Next we will define an equivalence relation on the set ${\bf E}=\{\, e_{xy}\, |\,  x < y\, \} $, which will be 
used to describe the structure of nonlinear lie-type derivations on $FI(X, \mathcal{R})$. 
Suppose that ${x_1},{x_2}, \cdots ,{x_m}\, (m\geq 2)$ are \textit{m different} elements in $X$. Then 
we say that $\{ {x_1},{x_2}, \cdots ,{x_m}\} $ forms a \emph{cycle} if: 
\begin{enumerate}
\item[(1)] either $m=2$ and ${x_1} \simeq {x_2}$, 
\item[(2)] or $m \ge 3$ and $\{ {x_1} \sim {x_2}, \cdots ,{x_i} \sim {x_{i + 1}}, \cdots ,{x_{m - 1}} \sim {x_m},{x_m} \sim {x_1}\} $.
\end{enumerate}

\begin{definition}
Let $e_{xy}, e_{uv}\in {\bf E}$, we denote ${e_{xy}} \approx {e_{uv}}$ if ${e_{xy}} = {e_{uv}}$, or there
is a cycle which contains both $x \sim y$ and $u \sim v$.
\end{definition}

\begin{lemma}{\rm \cite[Lemma 2.2]{Yang20212}}
The binary relation $\approx$ is an equivalence relation on the set ${\bf E}$.
\end{lemma}
In the following, we denote the equivalence classes of ${\bf E}$ under the equivalence 
relation $ \approx $ by $\{{\bf E}_i\, |\, i \in \mathcal{I}\, \}$, where
$\mathcal{I}$ is the index set.

\subsection{Nonlinear Lie-type derivations}

Let us introduce several kinds of (nonlinear) Lie-type derivations, which will
be used to describe nonlinear Lie-type derivation of $FI(X,\mathcal{R})$.  
\vspace{2mm}

\textbf{(I) Inner derivation}:

For an arbitrary element $\alpha\in FI(X,\mathcal{R})$, the map
$$
\begin{aligned}
{\rm ad}_{\alpha}: FI(X,\mathcal{R}) &\longrightarrow  FI(X,\mathcal{R}) \\
\beta &\longmapsto [\alpha,\beta]
\end{aligned}
$$
is called an \textit{inner derivation} of $FI(X, \mathcal{R})$. It is well known that each inner derivation is a derivation.

\vspace{2mm}

\textbf{(II) Central-valued map annihilating all $(n-1)$-th commutators}:

Let $\overline{FI(X,\mathcal{R})}$ be the set of $(n-1)$-th commutators, i.e., $$\overline{FI(X,\mathcal{R})}=\{\, p_{n}(x_1,\cdots , x_n)\, |\, x_1,\cdots , x_{n} \in FI(X,\mathcal{R})\, \}.$$
If $\kappa: FI(X,\mathcal{R})\longrightarrow  FI(X,\mathcal{R})$ is a map such that
\begin{eqnarray}
&\kappa(FI(X,\mathcal{R}))\subset \mathcal{C}(X,\mathcal{R}),\\
&\kappa(\overline{FI(X,\mathcal{R})})=0.
 \end{eqnarray}
It is easy to see that $\kappa$ is a nonlinear Lie $n$-derivation on $FI(X,\mathcal{R})$, which is called a 
\textit{central-valued} map annihilating all $(n-1)$-th commutators. 

\vspace{2mm}

\textbf{(III) Transtive induced derivation}:
In the following, we will use the notation $\leqslant$ to present the set $\{(x, y)\mid x\leq y \in X\}$ for convenience.
A map $f: \leqslant \longrightarrow \mathcal{R}$ is called \textit{transitive} if
$$
f(x,y)+f(y,z)=f(x,z)\, \, \,  \text{for all} \, \, \,  x,y,z \in X \ \text{such that} \ x\leq y\leq z.
$$
Note that every map $\sigma:X\longrightarrow \mathcal{R}$ determines a transitive map $(x,y) \longmapsto \sigma(x)-\sigma(y)$, 
and a transitive map of this form is said to be \textit{trivial}.
If $f: \leqslant \longrightarrow \mathcal{R}$ is a transitive map, then we can define an $\mathcal{R}$-linear 
map $L_f: FI(X,\mathcal{R})\longrightarrow FI(X,\mathcal{R})$ by
$$L_f(e_{xy})=f(x,y)e_{xy}\, \, \, \text{for all} \, \, \,  x\leq y \in X.$$
It is not hard to show that $L_f$ is a derivation on $FI(X,\mathcal{R})$ (see \cite{Xiao} for details), 
which is called a \textit{transitive induced derivation}. 
Furthermore, $L_f$ is an inner derivation if and only if $f$ is a trivial transitive map.

\textbf{(IV) Quasi-additive induced Lie $n$-derivation}: 

A map $f:\mathcal{R} \longrightarrow  \mathcal{R}$ is called an \textit{additive derivation} if
for all $r,s \in \mathcal{R}$ we have
\begin{eqnarray}
&f(r + s)  = f(r) + f(s), \\
& f(rs)  = f(r)s + rf(s).
\end{eqnarray}
For an additive derivation $f$ on $\mathcal{R}$, it is clear that $f(0)=f(1)=f(-1)=0$ and $f(-r) = -f(r)$ for all $r\in \mathcal{R}$. 
\begin{definition}
Let ${\bf F}:=\{\, {\bf f}_i \, \mid \,  i\in \mathcal{I}\, \}$ be a family of additive derivations on $\mathcal{R}$. A nonlinear 
Lie $n$-derivation $\psi_{\bf F}$ of $FI(X, \mathcal{R})$ is called an {\bf quasi-additive induced Lie $n$-derivation} 
associated to ${\bf F}$ if $\psi_{\bf F}$ leaves $\mathcal{D}(X, \mathcal{R})$ invariant, and 
\begin{eqnarray}
\psi_{\bf F}(re_{xy})=\mathbf{f}_i(r)e_{xy},\ e_{xy}\in {\bf E}_i,\ r\in \mathcal{R}.\label{eq1.8}
\end{eqnarray}
\end{definition}

It is not difficult to see that a central-valued map annihilating all $(n-1)$-th commutators is a special 
quasi-additive induced Lie $n$-derivation with $\mathbf{f}_i=0$ for all $i\in \mathcal{I}$. 
\begin{proposition}\label{p1.4}
A quasi-additive induced Lie $n$-derivation associated to ${\bf F}=\{\, {\bf f}_i \, | \, i\in \mathcal{I}\, \}$ is uniquely 
determined by ${\bf F}$ up to central-valued maps annihilating all $(n-1)$-th commutators.
\end{proposition}
\begin{proof}
Suppose $\psi_{\bf F}$ and $\phi_{\bf F}$ are two quasi-additive induced Lie  $n$-derivations associated to
${\bf F}=\{\, \mathbf{f}_i\, \mid \, i\in \mathcal{I}\, \}$ and $\Psi=\psi_{\bf F}-\phi_{\bf F}$. We need to prove that $\Psi$ is 
a central-valued map annihilating all $(n-1)$-th commutators. In the following, we define 
\begin{equation}\label{eq1.8}
\mathbbm{n}=\left\{
\begin{aligned} n, \quad & \text{if} \ n\geq 3; \\
3,\quad & \text{if}\ n=2. \end{aligned}\right.
\end{equation}
Note that if $n=2$, then $\Psi$ is a nonlinear Lie derivation, and hence also a Lie $3$-derivation. So $\Psi$ is a nonlinear Lie $\mathbbm{n}$-derivation.
The proof will be carried out by several steps.

Firstly, we will show that 
\begin{equation}\label{eq1.9}
\Psi(e_{xx})\in \mathcal{C}(X, \mathcal{R}) \, \, \, \text{for all}\, \, \,  x\in X.
\end{equation}
Since $\psi_{\bf F}$ and $\phi_{\bf F}$ leave $\mathcal{D}(X, \mathcal{R})$ invariant, $\Psi$ also leaves 
$\mathcal{D}(X, \mathcal{R})$ invariant. So we can assume that $$\Psi(e_{xx})=\sum_{z\in X}d_ze_{zz}.$$ 
 In order to prove $\Psi(e_{xx})\in \mathcal{C}(X, \mathcal{R})$, we only need to show that $d_r=d_s$ for all $r<s$. 
By \eqref{eq1.8}, for each pair $x<y$ we have 
\begin{equation}\label{eq1.10}
\Psi(ce_{xy})=\psi_{\bf F}(ce_{xy})-\phi_{\bf F}(ce_{xy})=0.
\end{equation}
Let $r<s$ in $X$. If $r,s\neq x$, then $[e_{rs},e_{xx}]=0$ and $p_{\mathbbm{n}}(e_{rs}, e_{xx},e_{ss},\cdots,e_{ss})=0$. We therefore obtain
\begin{equation}
\begin{split}
0=&\Psi(p_{\mathbbm{n}}(e_{rs}, e_{xx},e_{ss},\cdots,e_{ss}))\\
=&p_{\mathbbm{n}}(e_{rs}, \Psi(e_{xx}),e_{ss},\cdots,e_{ss})\\
=&(d_s-d_r)e_{rs}.
\end{split}
\end{equation}
This shows that $d_r=d_s$ for all $r<s$ and $r,s\neq x$. If $s=x$, applying $\Psi$ to $p_n(e_{rx}, e_{xx},\cdots,e_{xx})=e_{rx}$ 
gives
\begin{equation}
\begin{split}
0=&\Psi(p_n(e_{rx}, e_{xx},\cdots,e_{xx}))\\
=&p_n(e_{rx}, \Psi(e_{xx}),\cdots,e_{xx})+\cdots+p_n(e_{rx}, e_{xx},\cdots,\Psi(e_{xx}))\\
=&(n-1)(d_x-d_r)e_{rx}.
\end{split}
\end{equation}
So we get $d_r=d_x=d_s$, which is due to the fact that $\mathcal{R}$ is $2$ torsionfree and $(n-1)$-torsionfree. 
Similarly, one can show that $d_r=d_s$ if $r=x$. Thus far, we prove that $\Psi(e_{xx})\in \mathcal{C}(X, \mathcal{R})$ for all $x\in X.$

Secondly, let us show that
\begin{equation}
\Psi(\mathcal{D}(X, \mathcal{\mathcal{R}}))\subset \mathcal{C}(X, \mathcal{R}).\label{eq1.13}
\end{equation}
 Let $\mathfrak{r}=\sum\limits_{x\in X}r_x e_{xx}\in \mathcal{D}(X, \mathcal{R})$, where $r_x\in \mathcal{R}$ for all $x\in X$. 
Assume that $\Psi(\mathfrak{r})=\sum\limits_{x\in X}c_x e_{xx}\in \mathcal{D}(X, \mathcal{R})$, where $c_x\in \mathcal{R}$ for all $x\in X$.
 Applying $\Psi$ to $p_{\mathbbm{n}}(\sum_{x\in X}r_x e_{xx}, e_{xy},e_{yy},\cdots, e_{yy})=(r_x-r_y)e_{xy}$ 
 and considering \eqref{eq1.9}-\eqref{eq1.10}, we see that
$$
\begin{aligned}
0&=\Psi((r_x-r_y)e_{xy})\\
&= \Psi(p_{\mathbbm{n}}(\sum_{x\in X}r_x e_{xx}, e_{xy},e_{yy},\cdots, e_{yy}))\\
&=p_{\mathbbm{n}}(\Psi(\sum_{x\in X}r_x e_{xx}), e_{xy},e_{yy},\cdots,e_{yy})+p_{\mathbbm{n}}(\sum_{x\in X}r_x e_{xx}, \Psi(e_{xy}),e_{yy},\cdots,e_{yy})\\
&=p_{\mathbbm{n}}(\Psi(\sum_{x\in X}r_x e_{xx}), e_{xy},e_{yy},\cdots,e_{yy})\\
&=(c_x-c_y)e_{xy}.
\end{aligned}
$$
So $c_x=c_y$ for all $x<y$, and hence the relation \eqref{eq1.13} holds true.

Thirdly, we will prove that
\begin{equation}\label{eq1.14}
\Psi(FI(X, \mathcal{R}))\subset \mathcal{D}(X, \mathcal{R}).
\end{equation}
Let $\mathfrak{h}=\sum\limits_{x\leq y}h_{xy}e_{xy}$, where $h_{xy}\in R$ for all $x\leq y$. Assume that 
$\Psi(\mathfrak{h})=\sum\limits_{x\leq y}c_{xy}e_{xy}$. 
In order to get $\Psi(\mathfrak{h})\in  \mathcal{D}(X, \mathcal{R})$, we need to show that $c_{xy}=0$ for all $x<y$.  
If $x<y$ and $x$ is not equivalent to $y$, then we have 
\begin{equation}\label{eq1.15}
p_{\mathbbm{n}}(e_{xx}, \mathfrak{h}, e_{yy}, \cdots, e_{yy})=h_{xy}e_{xy}.
\end{equation}
Applying $\Psi$ to \eqref{eq1.15} yields that
$$
\begin{aligned}
0&=\Psi(h_{xy}e_{xy})\\
&=\Psi(p_{\mathbbm{n}}(e_{xx}, \mathfrak{h}, e_{yy}, \cdots, e_{yy}))\\
&=p_{\mathbbm{n}}(e_{xx}, \Psi(\mathfrak{h}),e_{yy}, \cdots, e_{yy})\\
&=c_{xy}e_{xy}.
\end{aligned}
$$
Thus we get $c_{xy}=0$. If $x\simeq y$,
then $p_{\mathbbm{n}}(e_{xx}, \mathfrak{h}, e_{yy},\cdots,e_{yy})=h_{xy}e_{xy}+h_{yx}e_{yx}$. Applying $\Psi$ to it we arrive at
\begin{equation}\label{eq1.16}
\begin{split}
\Psi(h_{xy}e_{xy}+h_{yx}e_{yx})&=\Psi(p_{\mathbbm{n}}(e_{xx}, \mathfrak{h}, e_{yy},\cdots,e_{yy}))\\
&=p_{\mathbbm{n}}(e_{xx}, \Psi(\mathfrak{h}),e_{yy},\cdots,e_{yy})\\
&=c_{xy}e_{xy}+c_{yx}e_{yx}.
\end{split}
\end{equation}
In view of \eqref{eq1.16}, applying $\Psi$ to $p_{\mathbbm{n}}(e_{yx},e_{xx},\cdots,e_{xx},h_{xy}e_{xy}+h_{yx}e_{yx})=h_{xy}(e_{yy}-e_{xx})$, we obtain
\begin{equation}
\begin{split}
\Psi(h_{xy}(e_{yy}-e_{xx})) &=\Psi(p_{\mathbbm{n}}(e_{yx},e_{xx},\cdots,e_{xx},h_{xy}e_{xy}+h_{yx}e_{yx}))\\
&=p_{\mathbbm{n}}(e_{yx},e_{xx},\cdots,e_{xx},\Psi(h_{xy}e_{xy}+h_{yx}e_{yx}))\\
&=[e_{yx},c_{xy}e_{xy}+c_{yx}e_{yx}]\\
&=c_{xy}(e_{yy}-e_{xx}).
\end{split}
\end{equation}
On the other hand, by \eqref{eq1.13} it follows that $c_{xy}(e_{yy}-e_{xx})=\Psi(h_{xy}(e_{yy}-e_{xx}))\in \mathcal{C}(X, \mathcal{R})$. 
This implies $c_{xy}=0$, and hence the relation \eqref{eq1.14} holds true.

Finally, we will prove that $\Psi(FI(X, \mathcal{R}))\subset \mathcal{C}(X, \mathcal{R})$. Let us choose an arbitrary $\mathfrak{s}\in FI(X, \mathcal{R})$. 
By \eqref{eq1.14} we can assume that $\Psi(\mathfrak{s})=\sum\limits_{x\in X} s_x e_{xx}\in \mathcal{D}(X, \mathcal{R})$, where $s_x\in \mathcal{R}$ for all $x\in X$.  
Then
$$
\begin{aligned}
\Psi(p_{\mathbbm{n}}(\mathfrak{s}, e_{xy},e_{yy},\cdots,e_{yy})) & =p_{\mathbbm{n}}(\Psi(\mathfrak{s}), e_{xy},e_{yy},\cdots, e_{yy})\\
&=(s_x-s_y)e_{xy}.
\end{aligned}
$$ 
On the other hand,  it follows from \eqref{eq1.14} that $\Psi(p_{\mathbbm{n}}(\mathfrak{s}, e_{xy},e_{yy},\cdots,e_{yy}))\in \mathcal{D}(X,\mathcal{R})$.
And hence $s_x-s_y=0$ for all $x<y$. This implies $\Psi(\mathfrak{s})\in \mathcal{C}(X, \mathcal{R})$ for all $\mathfrak{s}\in FI(X,\mathcal{R})$.

Since $\Psi$ is a nonlinear Lie $n$-derivation and $\Psi(FI(X, \mathcal{R}))\subset \mathcal{C}(X, \mathcal{R})$, 
it is clear that $\Psi$ annihilate all the $(n-1)$-th commutators.
We complete the proof of this proposition. 
\end{proof}

\textbf{(V) Additive induced Lie derivation}

For each $i \in \mathcal{I}$, let $V({{\bf{E}}_i}) = \{ \, x\in X\, |\, {e_{xv}}\ \mathrm{or}\ e_{ux}\in {{\bf{E}}_i}, u, v\in X\, \}$.  
Let $\mathcal{I}_x:=\{ i\, |\, x \in V({{\bf{E}}_i})\, \} $. For each $i \in \mathcal{I}_x$, let us set 
\begin{eqnarray*}
&V(x,i) = \{\,  y \in X\, |\, \exists {x_1} \in V({{\bf{E}}_i})\backslash \{ x\} ,{x_2}, \cdots ,{x_m} \in X\backslash \{ x\} \ {\rm{such\ that}}&\\
&x \sim {x_1},{x_1} \sim {x_2}, \cdots ,{x_{m - 1}} \sim {x_m},{x_m} \sim y\, \}.&
\end{eqnarray*}

As a more visual description, $y \in V(x, i)\setminus \{x\}$ means that $x$ is connected to $y$ by an acyclic graph
$\{ x \sim {x_1}\sim {x_2}\sim \cdots \sim {x_m} \sim y\}$ which must pass through a point ${x_1} \in V({{\bf{E}}_i})\setminus \{x\}$.

\begin{lemma}{\rm \cite[Lemma 2.3]{Yang20212}}
{\rm (1)} $V({{\bf{E}}_i}) \subset V(x,i)$ for all $x \in V({{\bf{E}}_i})$;

{\rm (2)} $V(x,i) \cap V(x,j) = \{ x\} \ {\rm{if}}\ i \ne j \in {\mathcal{I}_x}$.
\end{lemma}
Let ${\bf F}:=\{{\bf f}_i\mid i\in \mathcal{I}\}$ be a family of additive derivations satisfying the following assumption:
$\mathbf{f}_i=0$ if ${\bf E}_i\subset FI(X_j,\mathcal{R})$ for some $j\in \mathcal{J}$ such that $X_j$ is infinite.
Then we can define a map ${\psi _{\bf F}}:FI(X, \mathcal{R}) \longrightarrow FI(X, \mathcal{R})$ by
\begin{align}\label{eq1.18}
\psi _{\bf F}\left( {\sum\limits_{x \le y} {{r_{xy}}{e_{xy}}} } \right) = \sum\limits_{x \le y} \psi _{\bf F}(r_{xy}e_{xy}),\, \, \, \, r_{xy} \in \mathcal{R},
\end{align}
\begin{align}\label{eq1.19}
\psi _{\bf F}(r{e_{xy}}) = {\mathbf{f}_i}(r){e_{xy}},\quad{e_{xy}} \in {\bf E}_i,\ x < y,\ r \in \mathcal{R},
\end{align}
\begin{align}\label{eq1.20}
\psi_{\bf F}(re_{xx}) =  - \sum\limits_{i \in \mathcal{I}_x} \sum\limits_{y \in V(x,i)\backslash \{ x\} } \mathbf{f}_i (r) e_{yy} ,\quad r \in \mathcal{R}.
\end{align}

Note that ${\mathcal{I}_x} = \emptyset $ if and only if $x$ is an isolated point in $X$. This implies that
there is no element $y \in X\backslash \{ x\} $ such that $x\sim y$. If ${\mathcal{I}_x} = \emptyset $, then the right hand side
of \eqref{eq1.20} can be viewed as 0.  

\begin{lemma}{\rm \cite[Lemma 2.5]{Yang20212}}
$\psi_{\bf F}$ is a nonlinear Lie derivation on $FI(X, \mathcal{R})$.
\end{lemma}
The nonlinear Lie derivation $\psi_{\bf F}$ defined by \eqref{eq1.18}-\eqref{eq1.20} is called an {\bf additive induced Lie derivation} associated to ${\bf F}$. 
It is clear that $\psi_{\bf F}$ is additive, but not $\mathcal{R}$-linear in general.

\vspace{2mm}

\section{Inner-like Derivations of Finitary Incidence Algebras}
\label{xxsec2}

In this section, we will introduce a class of new derivations of finitary incidence algebras. 
Let us first see a proposition, which is a natural generalization of \cite[Theorem 1]{KN}.
\begin{proposition}
For each pre-ordered set  $X$, the incidence space $I(X,\mathcal{R})$ is a bimodule over $FI(X,\mathcal{R})$ with left and right module structures given by the convolution.
\end{proposition}
\begin{proof}
Let $\alpha\in FI(X,\mathcal{R})$, $\beta\in I(X,\mathcal{R})$. It is sufficient for us to show that  
$\alpha\beta$ and $\beta\alpha$ are well-defined elements in $I(X,\mathcal{R})$.
Since $\alpha\in FI(X,\mathcal{R})$, for each pair $x<y$ there are finite number of $x\leq z\leq y$ such that $\alpha(x,z)\neq 0$. 
It is straightforward to see that 
$$
\begin{aligned}
\alpha\beta&=\sum_{x\leq y}\left(\sum_{x\leq z\leq y}\alpha(x,z)\beta(z,y)\right)e_{xy},\\
\beta\alpha&=\sum_{x\leq y}\left(\sum_{x\leq z\leq y}\beta(x,z)\alpha(z,y)\right)e_{xy}
\end{aligned}
$$
are well-defined elements in $I(X,\mathcal{R})$.
\end{proof}

In what follows, for each element $\alpha\in I(X,\mathcal{R})$ we define 
\begin{equation}\label{eq2.1}
{\rm Ad}_{\alpha}(\beta):=\alpha\beta-\beta\alpha, \ \beta \in FI(X,\mathcal{R}).
\end{equation} 

\begin{proposition}\label{p2.2}
Let $\alpha\in I(X,\mathcal{R})$ such that ${\rm Ad}_{\alpha}(\beta)\in FI(X,\mathcal{R})$
for all $\beta\in FI(X,\mathcal{R})$. Then ${\rm Ad}_{\alpha}$ is a derivation of $FI(X,\mathcal{R})$.
\end{proposition}

\begin{proof}
Let $\beta,\gamma \in FI(X, \mathcal{R})$. Then we have
$$ 
\begin{aligned}
{\rm Ad}_{\alpha}(\beta\gamma)&=\alpha\beta\gamma-\beta\gamma\alpha\\
&=\alpha\beta\gamma-\beta\alpha\gamma+\beta\alpha\gamma-\beta\gamma\alpha\\
&={\rm Ad}_{\alpha}(\beta)\gamma+\beta{\rm Ad}_{\alpha}(\gamma).
\end{aligned}
$$
Thus ${\rm Ad}_{\alpha}$ is a derivation of $FI(X,\mathcal{R})$, which is due to the fact 
that ${\rm Ad}_{\alpha}$ is $\mathcal{R}$-linear. 
\end{proof}

\begin{definition}\label{d2.3}
Let $\alpha\in I(X,\mathcal{R})$ be an element satisfying the condition of Proposition \ref{p2.2}. 
Then ${\rm Ad}_{\alpha}$ is called an \textit{inner-like derivation} of $FI(X,\mathcal{R})$.
\end{definition}

It is clear that ${\rm Ad}_{\alpha}={\rm ad}_{\alpha}$ whenever $\alpha\in FI(X,\mathcal{R})$.
Next we will prove that there is an inner-like derivation of $FI(X,\mathcal{R})$ associated with 
each nonlinear Lie $n$-derivation of $FI(X, \mathcal{R})$. 

\begin{definition}
For a map $\Psi: FI(X,\mathcal{R})\longrightarrow FI(X,\mathcal{R})$ and a pair $x\leq y\in X$, let 
$\Psi_{x,y}: FI(X,\mathcal{R}) \longrightarrow \mathcal{R}$ be the function given by
\begin{equation}
\Psi_{xy}(\mathfrak{t})=\Psi(\mathfrak{t})(x,y)\ \mathrm{for\ all}\ \mathfrak{t}\in FI(X, \mathcal{R}).
\end{equation}
\end{definition}
\begin{remark}
For every map $\Psi:FI(X,\mathcal{R})\longrightarrow FI(X,\mathcal{R})$, we define $\Psi_{xy}=0$ 
if $x\nleq y$. Similarly, one can define $e_{xy}=0$ whenever $x\nleq y$. 
\end{remark}
In what follows, $L$ is a fixed nonlinear Lie $n$-derivation of $FI(X,\mathcal{R})$. Let  $e_{L}\in FI(X,\mathcal{R})$ be an element associated to $L$ given by 
\begin{equation}\label{eq2.3}
e_{L}=\sum_{x<y}L_{xy}(e_{yy})e_{xy}.
\end{equation}
Let us next show that ${\rm Ad}_{e_L}$ is an inner-like derivation of $FI(X,\mathcal{R})$.

\begin{lemma}
For all $x<y$ in $X$, we have 
\begin{equation}\label{eq2.4}
L_{xy}(e_{xx})=-L_{xy}(e_{yy}).
\end{equation}
\end{lemma}
\begin{proof}
Let $x<y\in X$ be a fixed pair. Applying $L$ to $p_n(e_{xx},e_{yy},\cdots, e_{yy})=0$ gives
\begin{equation}\label{eq2.5}
\begin{split}
0=&L(p_n(e_{xx},e_{yy},\cdots, e_{yy}))\\
=&p_n(L(e_{xx}),e_{yy},\cdots, e_{yy})+p_n(e_{xx},L(e_{yy}),\cdots, e_{yy})\cdots +p_n(e_{xx},e_{yy},\cdots, L(e_{yy}))\\
=&p_n(L(e_{xx}),e_{yy},\cdots,e_{yy})+p_n(e_{xx},L(e_{yy}),\cdots,e_{yy})\\
=&\sum_{u<y}L_{uy}(e_{xx})e_{uy}-(-1)^n\sum_{v>y}L_{yv}(e_{xx})e_{yv}+L_{xy}(e_{yy})e_{xy}-(-1)^nL_{yx}(e_{yy})e_{yx}.
\end{split}
\end{equation}
In view of the coefficient of $e_{xy}$ in \eqref{eq2.5}, we conclude that $L_{xy}(e_{xx})+L_{xy}(e_{yy})=0$. This implies \eqref{eq2.4}.
\end{proof}

For an arbitrary element $\beta\in FI(X,\mathcal{R})$, let us define
$$
\begin{aligned}
\beta|_u^v&=\sum_{u\leq x\leq y\leq v}\beta(x,y)e_{xy}\ \mathrm{for\ all}\ u\leq v\in X;\\
\beta^D&=\sum_{x\in X}\beta(x,x)e_{xx},\\
\beta^T&=\sum_{x<y}\beta(x,y)e_{xy}.
\end{aligned}
$$
It is not difficult to see that $\beta=\beta^D+\beta^T$.

\begin{lemma}\label{l2.7}
For all $\beta\in FI(X,\mathcal{R})$, we have ${\rm Ad}_{e_L}(\beta^T)\in FI(X,\mathcal{R})$.
\end{lemma}

\begin{proof}
According to the definition, we know that ${\rm Ad}_{e_L}(\beta^T)=e_L\beta^T-\beta^Te_L$. It is enough for us 
to show that $e_L\beta^T, \beta^Te_L\in FI(X,\mathcal{R})$.

Let $\beta$ be an arbitrary element in $FI(X,\mathcal{R})$. Then we have 
\begin{eqnarray}
e_L\beta^T =\sum_{x\leq y}[\sum_{x< z< y}L_{xz}(e_{zz})\beta(z,y)]e_{xy},\label{eq2.6}\\
\beta^T e_L =\sum_{x\leq y}[\sum_{x< z< y}\beta(x,z)L_{zy}(e_{yy})]e_{xy}.\label{eq2.7}
\end{eqnarray}
Let $u\leq v$ be a fixed pair of elements in $X$.  Since $\beta\in FI(X,\mathcal{R})$,
 there exists only finite many pairs $x<y$ such that $u\leq x<y\leq v$ and $\beta(x,y)\neq 0$, which are 
 denoted by $x_1<y_1,x_2<y_2,\cdots, x_k<y_k$. It follows from the relations \eqref{eq2.6}-\eqref{eq2.7} that 
\begin{eqnarray}
(e_L\beta^T)|_u^v  =\sum_{1\leq i\leq k}\sum_{u\leq x< x_i}L_{xx_i}(e_{x_ix_i})\beta(x_i,y_i)e_{xy_i},\label{eq2.8}\\
(\beta^T e_L)|_u^v =\sum_{1\leq i\leq k}\sum_{y_i< y\leq v}\beta(x_i,y_i)L_{y_iy}(e_{yy})e_{x_iy}.\label{eq2.9}
\end{eqnarray}
For each $1\leq i\leq k$, since $L(e_{x_ix_i})\in FI(X,\mathcal{R})$, there are only finite many elements $x_i^1,\cdots, x_i^{l_i}$ such that $u\leq x_i^j<x_i$ and 
$L_{x_i^jx_i}(e_{x_i^jx_i})=L(e_{x_ix_i})(x_i^j,x_i)\neq 0$ for all $1\leq j\leq l_i$. Thus \eqref{eq2.8} becomes 
\begin{equation}\label{eq2.10}
(e_L\beta^T)|_u^v=\sum_{1\leq i\leq k}\sum_{1\leq j\leq l_i}J_{x_i^{j}x_i}(e_{x_ix_i})\beta(x_i,y_i)e_{x_i^{j}y_i}.
\end{equation}
Similarly, for each $1\leq i\leq k$, let $\{y_i^1,\cdots y_i^{m_i}\}$ be the set of elements such that $y_i<y_i^r \leq v$ and  $L(e_{y_iy_i})(y_i,y_i^r)\neq 0$ for $1\leq r\leq m_i$.
By \eqref{eq2.4} and \eqref{eq2.9} we get 
\begin{equation}\label{eq2.11}
\begin{split}
(\beta^T e_L)|_u^v=&\sum_{1\leq i\leq k}\sum_{y_i< y\leq v}\beta(x_i,y_i)L_{y_iy}(e_{yy})e_{x_iy}\\
=&-\sum_{1\leq i\leq k}\sum_{y_i< y\leq v}\beta(x_i,y_i)L_{y_iy}(e_{y_iy_i})e_{x_iy}\\
=&-\sum_{1\leq i\leq k}\sum_{1\leq r\leq m_i}\beta(x_i,y_i)L_{y_iy_i^r}(e_{y_iy_i})e_{x_iy_i^r}.
\end{split}
\end{equation}
According to \eqref{eq2.10}-\eqref{eq2.11}, it is easy to see that $(e_L\beta^T)|_u^v$ and $(\beta^T e_L)|_u^v$ are 
both finite sums in terms of $\{e_{xy}|x\leq y\}$. We therefore have $e_L\beta^T, \beta^T e_L\in FI(X,\mathcal{R})$.
\end{proof}

\begin{lemma}\label{l2.8}
The nonlinear Lie $n$-derivation $L-{\rm Ad}_{e_L}$ leaves $\mathcal{D}(X,\mathcal{R})$ invariant.
\end{lemma}
\begin{proof}
Let $D=\sum\limits_{x\in X}d_xe_{xx}$ be an arbitrary element in $\mathcal{D}(X,\mathcal{R})$, where $d_x\in \mathcal{R}$ for all 
$x\in X$. We need to show that $(L-{\rm Ad}_{e_L})(D)\in \mathcal{D}(X,\mathcal{R})$.  
For a fixed pair $x< y$ in $X$, we have $p_n (D, e_{yy} ,e_{yy}, ... ,e_{yy})=0$. Applying $L$ to this identity we get
$$ 
\begin{aligned}
0&=L(p_n (D, e_{yy} ,e_{yy}, ... ,e_{yy})) \\
&=p_n (L(D), e_{yy} ,e_{yy}, ... ,e_{yy})+ p_n (D, L(e_{yy}) ,e_{yy}, ... ,e_{yy}).
\end{aligned}
$$
Multiplying $e_{xx}$ on the left and $e_{yy}$ on the right of this identity, we see that
$$L_{xy}(D) e_{xy}+ (d_x-d_y)L_{xy}(e_{yy})e_{xy}=0.$$
This shows hat 
\begin{equation}\label{eq2.12}
L_{xy}(D)=(d_y-d_x)L_{xy}(e_{yy})
\end{equation}	for all $x<y$ in $X$.
So we have 
$$
\begin{aligned}
(L-{\rm Ad}_{e_L})(D) &=L(D)-{\rm Ad}_{e_L}(D)\\
&=\sum_{x\leqslant y}L_{xy}(D)e_{xy}-\sum_{x<y}(d_y-d_x)L_{xy}(I_{xy})e_{xy}\\
&=\sum_{x\in X}L_{xx}(D)e_{xx}.
\end{aligned}
$$
Here the third identity follows from \eqref{eq2.12}. Thus we eventually arrive at 
$(L-{\rm Ad}_{e_L})(D)=\sum\limits_{x\in X}L_{xx}(D)e_{xx}\in \mathcal{D}(X, \mathcal{R})$.
\end{proof}

\begin{corollary}\label{c2.9}
For each element $D\in \mathcal{D}(X,\mathcal{R})$, we have ${\rm Ad}_{e_L}(D)\in FI(X,\mathcal{R})$.
\end{corollary}
\begin{proof}
For each element $D\in \mathcal{D}(X,\mathcal{R})$, we have $L(D)\in FI(X,\mathcal{R})$. 
By Lemma \ref{l2.8}, it follows that $(L-{\rm Ad}_{e_L})(D)\in \mathcal{D}(X, \mathcal{R})\subset FI(X,\mathcal{R})$.
So ${\rm Ad}_{e_L}(D)\in FI(X,\mathcal{R})$.
\end{proof}

\begin{proposition}
${\rm Ad}_{e_L}$ is an inner-like derivation of $FI(X,\mathcal{R})$.
\end{proposition}
\begin{proof}
By Definition \ref{d2.3}, we only need to show that ${\rm Ad}_{e_L}(\beta)\in FI(X,\mathcal{R})$ for all $\beta\in FI(X,\mathcal{R})$. Since $\beta=\beta^T+\beta^D$, we have ${\rm Ad}_{e_L}(\beta)={\rm Ad}_{e_L}(\beta^T)+{\rm Ad}_{e_L}(\beta^D)$.
By Lemma \ref{l2.7} and Corollary \ref{c2.9}, we conclude that ${\rm Ad}_{e_L}(\beta^T), {\rm Ad}_{e_L}(\beta^D)\in FI(X,\mathcal{R})$. 
This implies that ${\rm Ad}_{e_L}(\beta)\in FI(X,\mathcal{R})$ for all $\beta\in FI(X,\mathcal{R})$.
\end{proof}

\section{The Structure of Nonlinear Lie-Type Derivations}\label{xxsec3}

In this section, we will study nonlinear Lie-type derivations of the finitary incidence algebra $FI(X,\mathcal{R})$. 
Throughout this section, $n\geq 2$ is a fixed integer. 
Let $X= {\bigsqcup\limits_{i \in \mathcal{J}}}{X_i}$ be the decomposition of $X$ into the union of its connected components, 
and $\{{\bf E}_i\mid i\in \mathcal{I}\}$ the equivalence classes of ${\bf E}$ respect to relation $\approx$.
For each $i\in \mathcal{J}$, let $\pi_i:FI(X,\mathcal{R})\longrightarrow FI(X_i,\mathcal{R})$ be the canonical projective morphism.
The main result of this section is the following theorem.

\begin{theorem}\label{t3.1}
Every nonlinear Lie $n$-derivation on $FI(X, \mathcal{R})$ can be expressed as the sum of 
an inner-like derivation, a transitive induced derivation and a quasi-additive induced Lie $n$-derivation.
\end{theorem}
After some necessary preparations, we will prove the theorem at the end of this section.
In what follows, let $L$ be a fixed nonlinear Lie $n$-derivation, and let us define 
\begin{equation}
L^1:=L- {\rm Ad}_{e_L}.
\end{equation}
By Lemma \ref{l2.8}, we know that $L^1$ leaves $\mathcal{D}(X,\mathcal{R})$ invariant. So we can assume that  
\begin{equation}\label{e3.2}
L^1(e_{xx})=\sum_{z\in X}c^{x}_ze_{zz}, 
\end{equation}
where $c^x_z\in \mathcal{R},\ x, z\in X$. 

\begin{lemma}\label{l3.2}
	For each $x\in X$, we have $L^1(e_{xx})\in \mathcal{C}(X,\mathcal{R})$.
\end{lemma}
\begin{proof}
According to the relation \eqref{e3.2}, it is sufficient for us to prove that $c^x_r=c^x_s$ for all $r<s$ in $X$. In the following of the proof, let 
$\mathbbm{n}$ be the number associated to $n$ defined by \eqref{eq1.8}.
	
If $x\neq r$ and $x\neq s$,  then $[e_{xx}, e_{rs}]=0$ and  
\begin{equation}\label{e3.3}
p_{\mathbbm{n}} (e_{xx}, e_{rs},e_{ss}, ... , e_{ss})=0.
\end{equation}
 Applying $L^1$ to the both sides of \eqref{e3.3} yields
\begin{equation}\label{e3.4}
 p_\mathbbm{n} (L^1(e_{xx}), e_{rs},e_{ss}, ... , e_{ss})+p_\mathbbm{n} (e_{xx}, L^1 (e_{rs}),e_{ss}, ... , e_{ss})=0.
 \end{equation}
In view of the coefficient of $e_{rs}$ in \eqref{e3.4}, we assert that $c^x_r=c^x_s$.
	
If $x=s$, then $p_n ( e_{rs} ,e_{ss}, ..., ,e_{ss} )=e_{rs}$. Applying $L^1$ to this equality we see that
\begin{equation}\label{e3.5}
p_n ( L^1(e_{rs}) ,e_{ss}, ..., ,e_{ss} )+p_n ( e_{rs} ,L^1(e_{ss}), ..., ,e_{ss} )+...+p_n ( e_{rs} ,e_{ss}, ..., ,L^1(e_{ss}) )=L^1(e_{rs})
\end{equation}
 Considering the coefficients of $e_{rs}$ in both sides of \eqref{e3.5}, we get 
 $$L^1_{rs}(e_{rs})+(n-1)(c^x_s-c^x_r)=L^1_{rs}(e_{rs}).$$
Since $\mathcal{R}$ is $(n-1)$-torsionfree, we have  $c^x_s=c^x_r.$
 
In a similar manner, one can prove that $c^x_s=c^x_r$ whenever $x=r$.  We complete the proof of the lemma.
\end{proof}

To proceed our discussion, we need the following technique lemma.
\begin{lemma}\label{l3.3}
Let $\Psi$ be a nonlinear Lie $n$-derivation on $FI(X,\mathcal{R})$ and $\Psi_i:=\pi_i\circ \Psi|_{FI(X_i,\mathcal{R})}$, $i\in \mathcal{J}$. 
Then $\Psi_i$ is a nonlinear Lie $n$-derivation on $FI(X_i,\mathcal{R})$ and $\Psi_{i} ^{\bot}:=\Psi|_{FI(X_i,\mathcal{R})}-\Psi_i$ is a map from $FI(X_i,\mathcal{R})$ to 
$\bigoplus\limits_{j\in \mathcal{J}, j\neq i} \mathcal{C}(X_j,\mathcal{R})$.
\end{lemma}

\begin{proof}
Let $a_1, \cdots ,a_n \in FI(X_i,\mathcal{R})$. Then  $p_n (a_1, \cdots, a_n)\in FI(X_i,\mathcal{R})$. Furthermore, we get
$$
\begin{aligned}
 \Psi_i(p_n (a_1, \cdots, a_n))&=\pi_i\circ \Psi(p_n (a_1, \cdots, a_n))\\
&=\pi_i[p_n (\Psi(a_1), \cdots, a_n)+ \cdots + p_n (a_1, \cdots, \Psi(a_n ))] \\
&= p_n (\pi_i\circ \Psi(a_1), \cdots, a_n)+ \cdots + p_n (a_1, \cdots, \pi_i\circ \Psi(a_n ))   \\
&=p_n (\Psi_i(a_1), \cdots, a_n)+ \cdots +p_n (a_1, \cdots,  \Psi_i(a_n )).
\end{aligned}
$$
Thus $\Psi_i$ is a Lie $n$-derivation on $FI(X_i,\mathcal{R})$. 

Next we show that $\Psi_{i} ^{\bot}$ is a map from $FI(X_i,\mathcal{R})$ to 
$\bigoplus\limits_{i\neq j\in \mathcal{J}} \mathcal{C}(X_j,\mathcal{R})$. Suppose $a \in FI(X_i,\mathcal{R})$ and 
$b_1, \cdots, b_{n-1} \in FI(X_j, \mathcal{R})$ for some $j\in \mathcal{J}$ with $j\neq i$.
Applying $\Psi$ to $p_n (a, b_1, \cdots, b_{n-1})=0$ yields that
\begin{equation}\label{e3.6}
p_n (\Psi(a), b_1, \cdots, b_{n-1})+p_n (a, \Psi(b_1), \cdots, b_{n-1})=0.
\end{equation}
It is clear that $p_n (a,\Psi(b_1), \cdots, b_{n-1})=0$, which is due to the fact $a\in FI(X_i,\mathcal{R})$ and $b_k\in FI(X_j,\mathcal{R})$, $k=1,2,\cdots n-1$. 
We therefore get $p_n (\Psi(a), b_1, \cdots, b_{n-1})=0$. And hence 
\begin{equation}\label{e3.7}
\begin{split}
p_n (\Psi_i^{\bot}(a), b_1, \cdots, b_{n-1})=&p_n (\Psi_i+\Psi_i^{\bot})(a), b_1, \cdots, b_{n-1})\\
=&p_n (\Psi(a), b_1, \cdots, b_{n-1})\\
=&0.
\end{split}
\end{equation}
Note that $\Psi_i^{\bot}(a)\in \bigoplus\limits_{j\in \mathcal{J}\setminus \{i\}}FI(X_j,\mathcal{R})$. 
Then \eqref{e3.7} implies that 
$$\Psi_i^{\bot}(a)\in  \bigoplus_{j\in \mathcal{J}\setminus \{i\}}\mathcal{C}(X_j,\mathcal{R}).$$
This shows that $\Psi_{i} ^{\bot}$ is a map from $FI(X_i,\mathcal{R})$ to $\bigoplus_{i\neq j\in \mathcal{J}} \mathcal{C}(X_j,\mathcal{R})$.
\end{proof}

In what follows, we denote the cardinality of $X$ by $|X|$.  Suppose that $|X|=\infty$ if $X$ is infinite.
We say that a connected component $X_i$ of $X$ is {\it full} if $x\simeq y$ for all $x,y\in X_i$. 
Next we will prove that $L^1(\mathcal{S}_{xy})\subset \mathcal{S}_{xy}$ for all $x<y$ in $X$. 
However, we find that the main difficulty lies in the case that $|X_i|=2$ and $X_i$ is full for some $i\in \mathcal{J}$. 
 
\begin{lemma}\label{l3.4}
 If $X_i$ is full and $|X_i|=2$ for some $i\in \mathcal{J}$, then $L^1(\mathcal{S}_{xy})\subset \mathcal{S}_{xy}$ for all $x< y\in X_i$.
\end{lemma}
\begin{proof}
The proof of the lemma will be divided into four steps. In what follows, let 
$\mathbbm{n}$ be the number associated to $n$ defined by \eqref{eq1.8}.

{\bf Step 1}.  Since $X_i$ is full and $|X_i|=2$, we see that $X_i=\{x,y\}$ and $x\simeq y$. 
Applying $L^1$ to $p_\mathbbm{n}(e_{xx}, re_{xy}, e_{yy}, $ $\cdots, e_{yy})=re_{xy}$ and considering Lemma \ref{l3.2},  we obtain
\begin{equation}\label{e3.8}
\begin{split}
L^1(re_{xy})&=\, L^1(p_\mathbbm{n}(e_{xx},re_{xy},e_{yy},\cdots, e_{yy}))\\
&=\,p_\mathbbm{n}(e_{xx},L^1(re_{xy}),e_{yy},\cdots, e_{yy})\\
&=\, {L^1}_{xy}(re_{xy})e_{xy}+(-1)^{\mathbbm{n}-1}{L^1}_{yx}(re_{xy})e_{yx}.
\end{split}
\end{equation}
Let $\Phi=L^1+{\rm ad}_{{L^1}_{xy}(e_{xy})e_{yy}}$. Then it is not difficult to see that
\begin{equation}
\Phi(e_{xy})=(-1)^{\mathbbm{n}-1}{L^1}_{yx}(e_{xy})e_{yx}
\end{equation} and 
\begin{equation}
\Phi(re_{xy})=[{L^1}_{xy}(re_{xy})-{L^1}_{xy}(e_{xy})]e_{xy}+(-1)^{\mathbbm{n}-1}{L^1}_{yx}(e_{xy})e_{yx}.
\end{equation}
Let $f, g$ be the functions on $\mathcal{R}$ given by $f(r)={L^1}_{xy}(re_{xy})-{L^1}_{xy}(e_{xy}),$ $g(r)=(-1)^{\mathbbm{n}-1}{L^1}_{yx}(re_{xy})$ for all $r\in \mathcal{R}$. 
Then we get
\begin{equation}
\Phi(re_{xy})=f(r)e_{xy}+g(r)e_{yx}.
\end{equation}

{\bf Step 2}. Let us show that $f$ is an additive derivation on $\mathcal{R}$. 

For all $r,s\in \mathcal{R}$, we have 
\begin{eqnarray}
&p_\mathbbm{n}(re_{xx},se_{xy},e_{yy},\cdots,e_{yy})=rse_{xy}\label{e3.12}\\
&p_\mathbbm{n}(re_{xx},e_{xy},e_{yy},\cdots,e_{yy})=re_{xy}.\label{e3.13}
\end{eqnarray}
Applying $\Phi$ to the both sides of \eqref{e3.12} and using Lemma \ref{l3.2}, we get
\begin{equation}\label{e3.14}
\begin{split}
f(rs)e_{xy}+g(rs)e_{yx} &=\, p_\mathbbm{n}(\Phi(re_{xx}),se_{xy},e_{yy},\cdots,e_{yy})+p_\mathbbm{n}(re_{xx},\Phi(se_{xy}),e_{yy},\cdots,e_{yy})\\
&=\, p_\mathbbm{n}(\Phi(re_{xx}),se_{xy},e_{yy},\cdots,e_{yy})+p_\mathbbm{n}(re_{xx},f(s)e_{xy}+g(s)e_{yx},e_{yy},\cdots,e_{yy})\\
&=\, p_\mathbbm{n}(\Phi(re_{xx}),se_{xy},e_{yy},\cdots,e_{yy})+p_\mathbbm{n}(re_{xx},f(s)e_{xy},e_{yy},\cdots,e_{yy})\\
&\, \, \, \, \, \, \, +p_\mathbbm{n}(re_{xx},g(s)e_{yx},e_{yy},\cdots,e_{yy})\\
\end{split}
\end{equation}
Note that
\begin{eqnarray*}
&p_\mathbbm{n}(\Phi(re_{xx}),se_{xy},e_{yy},\cdots,e_{yy})\in \mathcal{S}_{xy},\\
&p_\mathbbm{n}(re_{xx},f(s)e_{xy},e_{yy},\cdots,e_{yy})\in \mathcal{S}_{xy},\\
&p_\mathbbm{n}(re_{xx},g(s)e_{yx},e_{yy},\cdots,e_{yy})\in \mathcal{S}_{yx}.
\end{eqnarray*}
By \eqref{e3.14}, we know that  
\begin{eqnarray}\label{e3.15}
p_\mathbbm{n}(\Phi(re_{xx}),se_{xy},e_{yy},\cdots,e_{yy})+p_\mathbbm{n}(re_{xx},f(s)e_{xy},e_{yy},\cdots,e_{yy})=f(rs)e_{xy}.
\end{eqnarray}
Applying $\Phi$ to the both sides of \eqref{e3.13} yields that
\begin{equation*}
\begin{split}
f(r)e_{xy}+g(r)e_{yx}=&p_\mathbbm{n}(\Phi(re_{xx}),e_{xy},e_{yy},\cdots,e_{yy})+p_\mathbbm{n}(re_{xx},\Phi(e_{xy}),e_{yy},\cdots,e_{yy})\\
=&p_\mathbbm{n}(\Phi(re_{xx}),e_{xy},e_{yy},\cdots,e_{yy})+p_\mathbbm{n}(re_{xx},g(1)e_{yx},e_{yy},\cdots,e_{yy})\\
\end{split}
\end{equation*}
This implies that
\begin{eqnarray}
p_\mathbbm{n}(\Phi(re_{xx}),e_{xy},e_{yy},\cdots,e_{yy})=f(r)e_{xy},\label{e3.16}
\end{eqnarray}
and
\begin{eqnarray}
p_\mathbbm{n}(re_{xx},g(1)e_{yx},e_{yy},\cdots,e_{yy})=g(r)e_{yx}.\label{e3.17}
\end{eqnarray}
Combining \eqref{e3.15} with \eqref{e3.16} gives that
\begin{equation}
f(rs)=rf(s)+sf(r).
\end{equation}

Next applying $\Phi$ to $p_\mathbbm{n}(e_{xx}-re_{xy},e_{xx}+se_{xy},e_{yy},\cdots,e_{yy})=(r+s)e_{xy}$, we obtain 
\begin{equation}\label{e3.19}
\begin{split}
&p_\mathbbm{n}(\Phi(e_{xx}-re_{xy}),e_{xx}+se_{xy},e_{yy},\cdots,e_{yy})+p_\mathbbm{n}(e_{xx}-re_{xy},\Phi(e_{xx}+se_{xy}),e_{yy},\cdots,e_{yy})\\
&=f(r+s)e_{xy}+g(r+s)e_{xy}.
\end{split}
\end{equation}
Applying $\Phi$ to $p_\mathbbm{n}(e_{xx}-re_{xy},e_{xy},e_{yy},\cdots,e_{yy})=e_{xy}$ and $p_\mathbbm{n}(e_{xx}-re_{xy},e_{xx},e_{yy},\cdots,e_{yy})=re_{xy}$, we get 
\begin{eqnarray*}
&p_\mathbbm{n}(\Phi(e_{xx}-re_{xy}),e_{xy},e_{yy},\cdots,e_{yy})=0,\\
&p_\mathbbm{n}(\Phi(e_{xx}-re_{xy}),e_{xx},e_{yy},\cdots,e_{yy})=f(r)e_{xy}+g(r)e_{yx}.
\end{eqnarray*}
Combining the previous two equalities, we arrive at
\begin{equation}\label{e3.20}
p_\mathbbm{n}(\Phi(e_{xx}-re_{xy}),e_{xx}+se_{xy},e_{yy},\cdots,e_{yy})=f(r)e_{xy}+g(r)e_{yx}.
\end{equation}
In a similar way,  one can show that 
\begin{equation}\label{e3.21}
p_\mathbbm{n}(e_{xx}-re_{xy},\Phi(e_{xx}+se_{xy}),e_{yy},\cdots,e_{yy})=f(s)e_{xy}+g(s)e_{yx}.
\end{equation}
Combining \eqref{e3.19} with \eqref{e3.21}, we conclude that $f(r+s)=f(r)+f(s).$
So $f$ is an additive derivation on $\mathcal{R}$.

Let ${\bf F}=\{{\bf f}_i\mid i\in \mathcal{I}\}$ such that ${\bf f}_i=f$ and ${\bf f}_j=0$ for $j\in \mathcal{I}, j\neq i$. 
Let $\psi_{\bf F}$ be the additive induced Lie $n$-derivation given 
by \eqref{eq1.18}-\eqref{eq1.20}, and $\Psi:=\Phi-\psi_{\bf F}$. It is straightforward to see that 
\begin{eqnarray}\label{e3.22}
\Psi(re_{xy})=g(r)e_{yx}.
\end{eqnarray}

{\bf Step 3}.  We will show that $\Psi(\mathfrak{d})\in \mathcal{C}(X, \mathcal{R})$ for all $\mathfrak{d}\in \mathcal{D}(X_i,\mathcal{R})$. 

Since $\Phi$ and $\psi_{\bf F}$ leaves $\mathcal{D}(X,\mathcal{R})$ invariant, $\Psi$ also leaves $\mathcal{D}(X,\mathcal{R})$ invariant.
Let $\mathfrak{d}=d_1e_{xx}+d_2e_{yy}$ be an arbitrary element in $\mathcal{D}(X_i,\mathcal{R})$. Then we have 
\begin{equation*}
p_\mathbbm{n}(\mathfrak{d}, e_{xy},e_{yy},\cdots, e_{yy})=(d_1-d_2)e_{xy}.
\end{equation*}
Applying $\Psi$ to the previous equality and using \eqref{e3.22} gives rise to
\begin{equation}
p_\mathbbm{n}(\Psi(\mathfrak{d}),e_{xy},e_{yy},\cdots, e_{yy})+p_\mathbbm{n}(\mathfrak{d},g(1)e_{yx},e_{yy},\cdots, e_{yy})=g(d_1-d_2)e_{yx}.
\end{equation} 
Thus we get 
\begin{equation}\label{e3.24}
p_\mathbbm{n}(\Psi(\mathfrak{d}),e_{xy},e_{yy},\cdots, e_{yy})=0.
\end{equation}
On the other hand, if $u<v\notin X_i$, then $p_\mathbbm{n}(\mathfrak{d}, e_{uv},e_{vv},\cdots, e_{vv})=0$. 
Applying $\Psi$ to it yields that
\begin{equation}\label{e3.25}
p_\mathbbm{n}(\Psi(\mathfrak{d}),e_{uv},e_{vv},\cdots, e_{vv})=0
\end{equation} 
for all $u<v\notin X_i$.
Combining \eqref{e3.24} with \eqref{e3.25}, we assert that $\Psi(\mathfrak{d})\in \mathcal{C}(X_i,\mathcal{R}) $ for all $\mathfrak{d}\in \mathcal{D}(X_i,\mathcal{R})$.

{\bf Step 4}. Let us prove that $g=0$. By \eqref{e3.17} it follows that
\begin{equation}
g(r)=(-1)^{\mathbbm{n}-1}rg(1),\ r\in \mathcal{R}.
\end{equation}
If $\mathbbm{n}$ is even, then $g(r)=-rg(1)$. Let us choose $r=1$. Then we see that $g(1)=0$. So $g(r)=-rg(1)=0$ for all $r\in \mathcal{R}$.
If $\mathbbm{n}$ is odd, then $g(r)=rg(1)$. Note that $\Psi(\mathcal{D}(X_i,\mathcal{R}))\subset \mathcal{C}(X,\mathcal{R})$. 
Applying $\Psi$ to $p_\mathbbm{n}(e_{xy},e_{yy},\cdots, e_{yy},e_{xy})=0$ and using the relation \eqref{e3.22}, we arrive at 
\begin{equation}
\begin{split}
0& =p_\mathbbm{n}(\Psi(e_{xy}),e_{yy},\cdots, e_{yy},e_{xy})+p_\mathbbm{n}(e_{xy},e_{yy},\cdots, e_{yy},\Psi(e_{xy}))\\
&=p_\mathbbm{n}(g(1)e_{yx},e_{yy},\cdots, e_{yy},e_{xy})+p_\mathbbm{n}(e_{xy},e_{yy},\cdots, e_{yy},g(1)e_{yx})\\
&=2g(1)(e_{xx}-e_{yy}).
\end{split}
\end{equation}
This implies that $g(1)=0$. And hence $g(r)=rg(1)=0$ for all $r\in \mathcal{R}$.

Since $g(r)=(-1)^{\mathbbm{n}-1}{L^1}_{yx}(re_{xy})$ for all $r\in \mathcal{R}$, by \eqref{e3.8} we assert that  
$L^1(re_{xy})=L^1_{xy}(re_{xy})e_{xy}\in \mathcal{S}_{xy}$.
\end{proof}

The following lemma is a general version of Lemma \ref{l3.4}.

\begin{lemma}\label{l3.5}
	For all $x<y$ in $X$, we have $L^1(\mathcal{S}_{xy})\subset \mathcal{S}_{xy}$.
\end{lemma}
\begin{proof}
Suppose that $x<y$ in $X_i$,  $i\in \mathcal{J}$. If $|X_i|=2$ and $X_i$ is full, then $L^1(\mathcal{S}_{xy})\subset \mathcal{S}_{xy}$ by Lemma \ref{l3.4}. 
Let $\mathbbm{n}$ be the number associated to $n$ defined by \eqref{eq1.8}.
If $|X_i|=2$ and $X_i$ is not full, then for all $r\in \mathcal{R}$ we have 
\begin{equation}\label{e3.28}
\begin{split}
L^1(re_{xy})& =L^1(p_\mathbbm{n}(e_{xx},re_{xy},e_{yy},\cdots, e_{yy}))\\
&=p_\mathbbm{n}(e_{xx},L^1(re_{xy}),e_{yy},\cdots, e_{yy})\\
&=L^1_{xy}(re_{xy})e_{xy}.
\end{split}
\end{equation}

Hence we also have $L^1(\mathcal{S}_{xy})\subset \mathcal{S}_{xy}$. 

Let us next consider the case of $|X_i|\geq 3$. If $y\nless x$, the relation \eqref{e3.28} still holds true, 
thus $L^1(re_{xy})\in \mathcal{S}_{xy}$. Now suppose that $y<x$, then $x\simeq y$.
Since $|X_i|\geq 3$, there exists an element $z\in X_i\setminus \{x,y\}$ such that  $z\sim x$. 
Without loss of generality, we may assume that $x< z$
(if $x>z$, the proof is similar). 
Note that \eqref{e3.8} is always true if $x\simeq y$. Applying $L^1$ to $p_\mathbbm{n}(re_{xy},e_{xz},e_{zz},...,e_{zz})=0$ 
yields
\begin{equation}
\begin{split}
&p_\mathbbm{n}(L^1_{xy}(re_{xy})e_{xy}+(-1)^{n-1}L^1_{yx}(re_{xy})e_{yx},e_{xz},e_{zz},...,e_{zz})\\
&+p_\mathbbm{n}(re_{xy},L^1_{xz}(e_{xz})e_{xz}+(-1)^{n-1}L^1_{zx}(e_{xz})e_{zx},e_{zz},...,e_{zz})=0.
\end{split}
\end{equation}
Considering the coefficient of $e_{yz}$ in above equality, we get 
$(-1)^{n+1}L^1_{yx}(re_{xy})=0$. By \eqref{e3.8} again, we see that 
$$
L^1(re_{xy})=L^1_{xy}(re_{xy})e_{xy}\in \mathcal{S}_{xy}.
$$
\end{proof}

Let $r_{xy}=L^1_{xy}(e_{xy})$ for all $x<y$. In view of Lemma \ref{l3.5}, we have
$$L^1(e_{xy})=r_{xy}e_{xy}\ \mathrm{for\ all}\ x<y \in X.$$ 

\begin{lemma}
The map $f:\le \longrightarrow \mathcal{R}$ defined by
\begin{equation}\label{e3.30}
	f(x,y)=\begin{cases}
	r_{xy},& \mathrm{if}\ x< y; \\
	0,& \mathrm{if}\ x=y.
	\end{cases}
\end{equation} 
is a transitive map.
\end{lemma}
\begin{proof}
Let $x<y<z$ in $X_i$ for some $i\in \mathcal{J}$. We need to prove
$$f(x,y)+f(y,z)=f(x,z).$$ In what follows, let $\mathbbm{n}$ be the number associated to $n$ defined by \eqref{eq1.8}.
	
If $x\neq z$, then applying $L^1$ to $p_\mathbbm{n}(e_{xy},e_{yz} ,e_{zz},...,e_{zz})=e_{xz}$ gives 
  $$
\begin{aligned}
r_{xz}e_{xz}&=L^1(e_{xz})\\
&=L^1(p_\mathbbm{n}(e_{xy},e_{yz} ,e_{zz},...,e_{zz}))\\
&=p_\mathbbm{n}(L^1(e_{xy}),e_{yz} ,e_{zz},...,e_{zz})+p_\mathbbm{n}(e_{xy},L^1(e_{yz}) ,e_{zz},...,e_{zz})\\
&=(r_{xy}+r_{yz})e_{xz}.
\end{aligned}
$$
So we get $f(x,y)+f(y,z)=f(x,z).$
	
If $x=z$,  then $x\simeq y$. In this case, we need to handle the following two cases: $|X_i|=2$ or $|X_i|\geq 3$.

If $|X_i|=2$, then $X_i=\{x,y\}$ is full. Applying $L^1$ to $p_\mathbbm{n}(e_{xy},e_{yy},...,e_{yy},e_{yx})=e_{xx}-e_{yy}$ results in
\begin{equation*}
\begin{split}
L^1(e_{xx}-e_{yy})&=L^1(p_\mathbbm{n}(e_{xy},e_{yy},...,e_{yy},e_{yx}) )\\
&=p_\mathbbm{n}(L^1(e_{xy}),e_{yy},...,e_{yy},e_{yx})+p_\mathbbm{n}(e_{xy},e_{yy},...,e_{yy},L^1(e_{yx}))\\
&=r_{xy}(e_{xx}-e_{yy})+r_{yz}(e_{xx}-e_{yy})\\
&=(r_{xy}+r_{yx})(e_{xx}-e_{yy}).
\end{split}
\end{equation*}
On the other hand, it is shown in the step (3) of the proof of Lemma \ref{l3.4} that  $\Psi=L^1+{\rm ad}_{r_{xy}e_{yy}}-\psi_{\bf F}$ maps $\mathcal{D}(X_i,\mathcal{R})$ to $\mathcal{C}(X_i,\mathcal{R})$. 
Since $({\rm ad}_{r_{xy}e_{yy}}-\psi_{F})(e_{xx}-e_{yy})=0$, we obtain 
\begin{equation*}
L^1(e_{xx}-e_{yy})=(r_{xy}+r_{yx})(e_{xx}-e_{yy})\in \mathcal{C}(X_i,\mathcal{R}).
\end{equation*} 
This forces that $r_{xy}+r_{yx}=0$. And hence $f(x,y)+f(y,x)=0=f(x,x).$

If $|X_i|\geq 3$, there exists $w\in X_i\setminus \{x,y\}$ such that $w\thicksim x$. 
Assume that $x<w$ (the proof for the case of $w<x$ is similar). Applying $L^1$ to 
$$
p_\mathbbm{n}(e_{xy},e_{yx} ,e_{xw}, e_{ww},... , e_{ww})=e_{xw}
$$ 
produces
$$
\begin{aligned}
r_{xw}e_{xw} &= L^1(p_\mathbbm{n} (e_{xy},e_{yx} ,e_{xw}, e_{ww},... , e_{ww}))\\
&=p_\mathbbm{n} (L^1(e_{xy}),e_{yx} ,e_{xw},... , e_{ww})+ p_\mathbbm{n} (e_{xy},L^1(e_{yx}) ,e_{xw},... , e_{ww})+p_\mathbbm{n} (e_{xy},e_{yx} ,L^1(e_{xw}), ... , e_{ww}) \\
&=(r_{xy}+r_{yx})e_{xw}+r_{xw}e_{xw} .
\end{aligned}
$$
This implies that $r_{xy}+r_{yx}=0$. In this case, we also have $f(x,y)+f(y,x)=0=f(x,x)$.
\end{proof}

In the following, let $L_f$ be the transitive induced derivation of $FI(X,\mathcal{R})$ associated 
with $f$ given by \eqref{e3.30}, and let us define
$$L^2:=L^1-L_f.$$
Since $L_f(e_{xy})=f(x,y)e_{xy}$, we know that  $L^2(\mathcal{S}_{xy})\subset \mathcal{S}_{xy}$ 
and $L^2(e_{xy})=0$ for all $x< y$. Furthermore, $L^2$ leaves $\mathcal{D}(X,\mathcal{R})$ invariant by 
Lemma \ref{l2.8} and the definition of $L_f$, and $L^2(e_{xx})=L^1(e_{xx}) \in \mathcal{C}(X,\mathcal{R})$ for all $x\in X$ by Lemma \ref{l3.2}.
For each pair $x<y\in X$, let $f_{xy}:\mathcal{R}\longrightarrow \mathcal{R}$ be the function given by $f_{xy}(r)=L^2_{xy}(re_{xy})$ for all $r\in \mathcal{R}$. 
Thus we have 
\begin{equation}\label{e3.31}
L^2(re_{xy})=f_{xy}(r)e_{xy}. 
\end{equation}

\begin{lemma}\label{l3.7}
For all $x<y, u<v$ in $X$, if $e_{xy}\thickapprox e_{uv}$, then we have $f_{xy}=f_{uv}$.
\end{lemma}

\begin{proof} 
Let $\mathbbm{n}$ be the number associated to $n$ defined by \eqref{eq1.8}.
If $x=u$, $y=v$, i.e. $e_{xy}= e_{uv}$, the result of the lemma is trivial. If $x=v$ and $y=u$, then $x\simeq y$. 
Applying $L^2$ to $p_\mathbbm{n}(re_{xy},e_{yy},...,e_{yy},e_{yx})=p_\mathbbm{n}(e_{xy},e_{yy},...,e_{yy},re_{yx})$ results in
$$
f_{xy}(r)=f_{yx}(r)\ {\rm{for\ all}}\ r\in \mathcal{R}.
$$ 
We therefore have $f_{xy}=f_{yx}$ if $x\simeq y$.
	
Now suppose that $\{x,y\}\neq \{u,v\}$, then there exists a cycle $C=\{x_1\thicksim x_2,\cdots, x_i\thicksim x_{i+1},\cdots, x_{m-1}\thicksim x_{m}, x_{m}\thicksim x_1\}$ 
containing both $x\thicksim y$ and $u\thicksim v$, where $m\geq 3$. In the following,  for $s\thicksim t\in X$ and $s\neq t$, let us define
\begin{equation}
	f_{\overline{st}}=\begin{cases}
	f_{st},& \mathrm{if}\ s< t; \\
	f_{ts},& \mathrm{if}\ t<s.
	\end{cases}
\end{equation}

In order to prove $f_{xy}=f_{uv}$, we only need to show that $f_{\overline{x_{i-1} x_i}}=f_{\overline{x_ix_{i+1}}}$ for $i\in \{2,\cdots, m\}$, where $x_{m+1}:=x_1$. 
Since  $L_f(re_{zz})=0$ for all $r\in \mathcal{R}$ and $z\in X$, we get $L^2(re_{x_i x_i})=L^1(re_{x_i x_i}) \in \mathcal{D}(X, \mathcal{R})$. Let $i\in \{1,2,\cdots, m\}$ 
be a fixed number, we can assume that
$$
L^2(re_{x_i x_i})=\sum_{w\in X}g_w(r)e_{ww},
$$
where $\{\, g_w\, |\,  w\in X\, \}$ are functions on $\mathcal{R}$. Let $1\leq k<l\leq m$ and $k,l\neq i$. 
Applying $L^2$ to $p_\mathbbm{n}(re_{x_i x_i},e_{x_kx_l},$ $e_{x_lx_l},...,e_{x_lx_l})=0$ (or $p_\mathbbm{n}(re_{x_i x_i},e_{x_lx_k},$ $e_{x_kx_k},...,e_{x_kx_k})=0$ if $x_l<x_k$), 
we arrive at
\begin{equation}
	(g_{x_k}(r)-g_{x_l}(r))e_{x_kx_l}=0 \ (\mathrm{or}\ (g_{x_l}(r)-g_{x_k}(r))e_{x_lx_k}=0).
\end{equation}
So we have $g_{x_k}(r)=g_{x_l}$ for all $1\leq k<l\leq m$ and $k,l\neq i$. This forces that
\begin{equation}\label{e3.34}
g_{x_{i-1}}=g_{x_{i-2}}=\cdots =g_{x_1}=g_{x_m}=g_{x_{m-1}}=\cdots =g_{x_{i+1}}.
\end{equation}
	
Let us next show that  $f_{\overline{x_{i-1} x_i}}=f_{\overline{x_i x_{i+1}}}$ for all $1<i\leq m$. We need to consider the following four cases.
	
{\bf Case} 1. $x_{i-1}<x_i<x_{i+1}$. For any $r\in \mathcal{R}$, applying $L^2$ to
$$
p_\mathbbm{n}(re_{x_{i-1}x_i},e_{x_i x_{i+1}},e_{x_{i+1}  x_{i+1}}, \cdots ,e_{x_{i+1}  x_{i+1}})=p_\mathbbm{n}(e_{x_{i-1}x_i},re_{x_i x_{i+1}},e_{x_{i+1}  x_{i+1}}, \cdots ,e_{x_{i+1}  x_{i+1}}),
$$ 
we see that ${x_{i-1}x_i}(r)=f_{x_i x_{i+1}}(r)$.
	
{\bf Case} 2. $x_{i-1}>x_i>x_{i+1}$. In a similar manner as in {\bf Case} 1, we assert that $f_{x_i x_{i-1}}=f_{x_{i+1}  x_i}$.
	
{\bf Case} 3. When $x_{i-1}<x_i, x_i> x_{i+1}$.  Applying $L^2$ to $p_\mathbbm{n}(e_{x_{i-1} x_i},re_{x_i x_i},\cdots, e_{x_i  x_{i}})=re_{x_{i-1} x_i}$
yields that
\begin{equation}\label{e3.35}
g_{x_i}(r)-g_{x_{i-1}}(r)=f_{x_{i-1} x_i}(r).
\end{equation}
 Similarly, applying $L^2$ to $p_\mathbbm{n}(e_{x_{i+1} x_i},re_{x_i x_i},\cdots, e_{x_i  x_{i}})=re_{x_{i+1} x_i}$ we obtain 
\begin{equation}\label{e3.36}
g_{x_i}(r)-g_{x_{i+1}}(r)=f_{x_{i+1} x_i}(r).
\end{equation}
By the relations \eqref{e3.34}-\eqref{e3.36} it follows that $f_{x_{i-1}x_i}=f_{x_{i+1}x_i}$.
	
{\bf Case} 4. When  $x_{i-1}>x_i$ and $x_i< x_{i+1}$.  The proof is similar to the proof of {\bf Case} 3.

Thus far, we complete the proof of this lemma.
\end{proof}

By Lemma \ref{l3.7}, for each $i\in \mathcal{I}$  we can define a function $\mathbf{f}_i=f_{xy}$, where $e_{xy}\in \mathbf{E}_i$. 
Then we have 
\begin{equation}\label{e3.37}
L^2(re_{xy})=\mathbf{f}_i(r)e_{xy}, \ e_{xy}\in {\bf E}_i.
\end{equation}

\begin{lemma}\label{l3.8}
For each $i\in \mathcal{I}$, the map $\mathbf{f}_i$ is an additive derivation on $\mathcal{R}$.
\end{lemma}
\begin{proof}
Let $\mathbbm{n}$ be the number associated to $n$ defined by \eqref{eq1.8}. 
Let $e_{xy}\in \mathbf{E}_i$. Then for all $r, s\in \mathcal{R}$, we have
\begin{equation}\label{e3.38}
p_{\mathbbm{n}}(e_{xx}-re_{xy},e_{xx}+se_{xy} , e_{yy},\cdots , e_{yy})=(r+s)e_{xy}. 
\end{equation}
Applying $L^2$ to both sides of \eqref{e3.38} gives that 
\begin{equation}\label{e3.39}
p_{\mathbbm{n}}(L^2(e_{xx}-re_{xy}),e_{xx}+se_{xy} , e_{yy},\cdots , e_{yy})+p_{\mathbbm{n}}(e_{xx}-re_{xy},L^2(e_{xx}+se_{xy}) , e_{yy},\cdots , e_{yy})=\mathbf{f}_i(r+s)e_{xy}.
\end{equation}
On the other hand, we see that
\begin{equation}
\begin{split}
&p_{\mathbbm{n}}(L^2(e_{xx}-re_{xy}),e_{xx}+se_{xy} , e_{yy},\cdots , e_{yy})\\
&=p_{\mathbbm{n}}(L^2(e_{xx}-re_{xy}),e_{xx} , e_{yy},\cdots , e_{yy})+sp_{\mathbbm{n}}(L^2(e_{xx}-re_{xy}),e_{xy} , e_{yy},\cdots , e_{yy})\\
&=L^2(p_{\mathbbm{n}}(e_{xx}-re_{xy},e_{xx} , e_{yy},\cdots , e_{yy}))+sL^2(p_{\mathbbm{n}}(e_{xx}-re_{xy},e_{xy} , e_{yy},\cdots , e_{yy}))\\
&=L^2(re_{xy})+sL^2(e_{xy})\\
&=\mathbf{f}_i(r)e_{xy}.
\end{split}
\end{equation}
Similarly, we also have
\begin{equation}
\begin{split}
&p_{\mathbbm{n}}(e_{xx}-re_{xy},L^2(e_{xx}+se_{xy}) , e_{yy},\cdots , e_{yy})\\
&=p_{\mathbbm{n}}(e_{xx},L^2(e_{xx}+se_{xy}), e_{yy},\cdots, e_{yy})-rp_{\mathbbm{n}}(e_{xy},L^2(e_{xx}+se_{xy}), e_{yy},\cdots , e_{yy})\\
&=L^2(p_{\mathbbm{n}}(e_{xx},e_{xx}+se_{xy}, e_{yy},\cdots, e_{yy}))-rL^2(p_{\mathbbm{n}}(e_{xy},e_{xx}+se_{xy}, e_{yy},\cdots , e_{yy}))\\
&=L^2(se_{xy})+rL^2(-e_{xy})\\
&=\mathbf{f}_i(r)e_{xy}.
\end{split}
\end{equation}
The last equality is due to the relation $L^2(-e_{xy})=L^2(p_{\mathbbm{n}} (e_{xy},e_{xx}, e_{yy},\cdots , e_{yy}))=0$. 	
We therefore get $$\mathbf{f}_i(r+s)=\mathbf{f}_i(r)+\mathbf{f}_i(s).$$
	
Next, applying $L^2$ to $p_{\mathbbm{n}}(re_{xx},se_{xy}, e_{yy}, \cdots, e_{yy}) =rs e_{xy}$ results in 
\begin{equation}\label{e3.42}
p_{\mathbbm{n}}(L^2(re_{xx}),se_{xy}, e_{yy}, \cdots, e_{yy})+p_{\mathbbm{n}}(re_{xx},L^2(se_{xy}), e_{yy}, \cdots, e_{yy})=\mathbf{f}_i(rs)e_{xy}.
\end{equation}
Furthermore, we see that  
\begin{equation}\label{e3.43}
	\begin{split}
	p_{\mathbbm{n}}(L^2(re_{xx}),se_{xy}, e_{yy}, \cdots, e_{yy})=& sp_{\mathbbm{n}}(L^2(re_{xx}),e_{xy}, e_{yy}, \cdots, e_{yy})\\
	=& sL^2(p_{\mathbbm{n}}(re_{xx},e_{xy}, e_{yy}, \cdots, e_{yy}))\\
	=& s \mathbf{f}_i(r)e_{xy}.
	\end{split}
\end{equation}
and 
\begin{equation}\label{e3.44}
p_{\mathbbm{n}}(re_{xx},L^2(se_{xy}), e_{yy}, \cdots, e_{yy})=rf_i(s)e_{xy}.
\end{equation}
In view of the relations \eqref{e3.42}-\eqref{e3.44}, we conclude that $\mathbf{f}_i(rs)=\mathbf{f}_i(r)s+r\mathbf{f}_i(s)$.

The above discussion shows that $\mathbf{f}_i$ is an additive derivation on $\mathcal{R}$ for each $i\in \mathcal{I}$.
\end{proof}

{\bf Proof of Theorem 3.1}.
Let ${\bf F}:=\{\mathbf{f}_i|i\in\mathcal{I}\}$, where $\mathbf{f}_i$ is the function determined by \eqref{e3.37} for each $i\in \mathcal{I}$. 
Since $L^2$ leaves $\mathcal{D}(X,\mathcal{R})$ invariant, by Lemma \ref{l3.8}, we say that $L^2$ is a 
quasi-additive induced Lie $n$-derivation associated with ${\bf F}$. 
In view of the fact $L^2=L-{\rm Ad}_{e_L}-L_f$, we have $L={\rm Ad}_{e_L}+L_f+L^2$, where $e_L$ is given 
by \eqref{eq2.3} and $L_f$ is a transitive induced derivation associated with $f$ given by \eqref{e3.30}.
 \hfill $\Box$

If $X$ is finite, the structure of a nonlinear lie-type derivation can be given in a more explicit way. 
Keeping the notations as above, we have the following theorem.

\begin{theorem}
If $X$ is a finite pre-ordered set, then $L={\rm ad}_{e_L}+L_f+\psi_{\bf F}+\kappa$, where $\psi_{\bf F}$ 
is the additive induced Lie $n$-derivation determined by \eqref{eq1.18}-\eqref{eq1.20},
$\kappa$ is a central-valued map annihilating all $(n-1)$-commutators.
\end{theorem}
\begin{proof}
Since $X$ is finite, we have $FI(X,\mathcal{R})=I(X,\mathcal{R})$ and ${\rm Ad}_{e_L}={\rm ad}_{e_L}$. 
By Lemma \ref{l3.8}, we know that $L^2$ is an additive induced Lie $n$-derivation associated with ${\bf F}:=\{\mathbf{f}_i|i\in\mathcal{I}\}$. 
By Proposition \ref{p1.4}, $L^2-\psi_F$ is a central-valued map annihilating all $(n-1)$-th commutators. 
Let us write $\kappa=L^2-\psi_{\bf F}$.  
Then $L={\rm ad}_{e_L}+L_f+\psi_{\bf F}+\kappa$, which is the desired result.
\end{proof}

\section{Nonlinear Lie-type Derivations with Proper Form}
\label{xxsec4}
Thus far, we prove that every nonlinear Lie $n$-derivation on $FI(X,\mathcal{R})$ is the 
sum of an inner-like derivation, a transitive 
induced derivation and a quasi-additive induced Lie $n$-derivation (cf. Theorem \ref{t3.1}).
In this section, we will give a sufficient and necessary condition such that each nonlinear lie-type derivation 
of a finitary incidence algebra is of proper form. This provides an affirmative answer to \cite[Problem 2.15]{KhrypchenkoWei}.
Throughout this section, we keep the notions and notations in Section \ref{xxsec3}. 

Let us first establish the following proposition. 

\begin{proposition}\label{p4.1}
Let $\psi$ be a quasi-additive induced Lie $n$-derivation of $FI(X,\mathcal{R})$ associated with 
${\bf F}:=\{\mathbf{f}_i\mid i\in \mathcal{I}\}$. 
Then $\psi$ is proper if and only if $\mathbf{f}_i=\mathbf{f}_j$ whenever $\mathbf{E}_i$ and $\mathbf{E}_j$ are 
contained in the same connected component of $FI(X,\mathcal{R})$.
\end{proposition}

\begin{proof}
Firstly, we assume that $\psi$ is proper. Then $\psi=D+\kappa$, where $D$ is an 
additive nonlinear derivation and $\kappa$ is a central-valued map annihilating all $(n-1)$-th commutators of $FI(X,\mathcal{R})$.
Since $\psi$ be a quasi-additive induced Lie $n$-derivation associated with ${\bf F}:=\{\mathbf{f}_i\mid i\in \mathcal{I}\}$, 
we have $\psi(re_{xy})=\mathbf{f}_i(r)e_{xy}$ for all $r\in \mathcal{R}$, $e_{xy}\in \mathbf{E}_i, i\in \mathcal{I}.$ 
Note that $\kappa$ is a central-valued map. Then for all $r\in\mathcal{R}$, we get 
\begin{equation}
D(re_{xy})=\mathbf{f}_i(r)e_{xy}, \ e_{xy}\in \mathbf{E}_i,\  i\in \mathcal{I}.\label{eq4.1}
\end{equation}
It is clear that $D$ leaves $\mathcal{D}(X,\mathcal{R})$ invariant. Then for each $x\in X$, we can 
assume $D(re_{xx})=\sum\limits_{u\in X} f^x_{u}(r)e_{uu}$, where  $\{\, f^x_u\, |\, x,u\in X\, \}$ are functions on $\mathcal{R}$. 
For each pair $x\neq y$ in $X$, we have
\begin{equation*}
0=D(re_{xx}e_{yy})=D(re_{xx})e_{yy}+re_{xx}D(e_{yy})=f^x_y(r)e_{yy}+rf^y_x(1)e_{xx}.
\end{equation*}
This implies that $f^x_y=0$ if $x\neq y$. And hence $D(re_{xx})=f^x_x(r)e_{xx}$ for all $x\in X$, $r\in \mathcal{R}$.
By invoking the relation \eqref{eq4.1}, we see that $D(e_{xy})=0$ for $x<y$. We therefore conclude 
\begin{eqnarray}
&&D(re_{xy})=D(re_{xx}e_{xy})=D(re_{xx})e_{xy}+re_{xx}D(e_{xy})=f_x^x(r)e_{xy},\label{eq4.2}\\
&&D(re_{xy})=D(e_{xy}re_{yy})=D(e_{xy})re_{yy}+e_{xy}D(re_{yy})=f_y^y(r)e_{xy}.\label{eq4.3}
\end{eqnarray}
This forces that $f^x_x=f^y_y$ for all $x<y$. Then for each connected component $X_l$, $l\in \mathcal{J}$, we obtain
\begin{equation}\label{eq4.4}
f_u^u=f^v_v, \ u,v\in X_l.
\end{equation}

Now suppose $e_{xy}\in \mathbf{E}_i$, then by \eqref{eq4.1}-\eqref{eq4.2} we have 
\begin{equation}\label{eq4.5}
f^x_x(r)e_{xy}=D(re_{xy})=\mathbf{f}_i(r)e_{xy}.
\end{equation}
Combining \eqref{eq4.4} with \eqref{eq4.5} gives $\mathbf{f}_i=\mathbf{f}_j$ whenever $\mathbf{E}_i$ and $\mathbf{E}_j$ 
are contained in the same connected component of $FI(X,\mathcal{R})$.

Conversely, suppose that $\mathbf{f}_k=\mathbf{f}_l$ if $\mathbf{E}_k$ and $\mathbf{E}_l$ are 
contained in the same connected component of $FI(X,\mathcal{R})$. 
Then for each $j\in \mathcal{J}$, let $f_j=\mathbf{f}_i$ such 
that $\mathbf{E}_i$ is contained in $FI(X_j,\mathcal{R})$.
Then we can define a map $\Psi:FI(X,\mathcal{R})\longrightarrow FI(X,\mathcal{R})$ by

\begin{eqnarray}
&{\Psi}\left( {\sum\limits_{x \le y} {{r_{xy}}{e_{xy}}} } \right) = \sum\limits_{x \le y} {{\Psi}({r_{xy}}{e_{xy}})} ,\quad{r_{xy}} \in \mathcal{R},\label{eq4.6}\\
&{\Psi }(r{e_{xy}}) = f_j(r){e_{xy}},\ r\in\mathcal{R},\ x\leq y\in X_j,\ j\in \mathcal{J}.\label{eq4.7}
\end{eqnarray}
Next we will show that $\Psi$ is an additive nonlinear derivation of $FI(X,\mathcal{R})$. By the relations
\eqref{eq4.6}-\eqref{eq4.7}, it is easy to see that $\Psi$ is additive. It is sufficient for us to prove that 
\begin{equation}\label{eq4.8}
\Psi(re_{xy}se_{uv})=\Psi(re_{xy})se_{uv}+re_{xy}\Psi(se_{uv}),\ \mathrm{for}\ \mathrm{all}\ r,s\in \mathcal{R},\ x\leq y,\ u\leq v.
\end{equation}
If $y\neq u$, then it is clear that the both sides of \eqref{eq4.8} equal to $0$. If $y=u$, then we can 
assume that $x,y,v\in X_j$ for some $j\in \mathcal{J}$. Thus we obtain
$$
\begin{aligned}
\Psi(re_{xy}se_{uv})&= \Psi(rse_{xv})\\
&= f_j(rs)e_{xv}\\
&= (f_j(r)s+rf_j(s))e_{xv}\\
&=f_j(r)se_{xy}e_{uv}+rf_j(s)e_{xy}e_{uv}\\
&= \Psi(re_{xy})se_{uv}+re_{xy}\Psi(se_{uv}).
\end{aligned}
$$
This implies that $\Psi$ is an additive nonlinear derivation of $FI(X,\mathcal{R})$. It follows from  
Proposition \ref{p1.4} that $\kappa:=\psi-\Psi$ is a central-valued map annihilating all $(n-1)$-th commutators.
So $\psi=\Psi+\kappa$ has the proper form.
\end{proof}

To move on, we need more notations. Recall that for each $i\in \mathcal{I}$, 
\begin{equation}
V(\mathbf{E}_i)=\{\, x\, |\,  e_{xv}\ \mathrm{or}\ e_{ux}\in E_i\ {\rm{for\ some}} \ u,v\in X \, \}.
\end{equation}
Suppose $V(\mathbf{E}_i)\subset X_l$ for some $l\in \mathcal{J}$. For each $x\in V(\mathbf{E}_i)$, let $V_x$ be 
the set consisting of $x$ and the elements $y\in X_l\setminus V(\mathbf{E}_i)$ such that 
$x \sim {x_1},{x_1} \sim {x_2}, \cdots ,{x_{m - 1}} \sim {x_m},{x_m} \sim y$ for some 
${x_1} ,{x_2}, \cdots ,{x_m} \notin V(\mathbf{E}_i)$, $m\geq 0$.  Here $m=0$ is well understood as $x\sim y$.

\begin{lemma}\label{l4.2}
\begin{itemize}
\item[(1)] If $x,y\in V(\mathbf{E}_i)$ and $x\neq y$, then $V_x\cap V_y= \emptyset$.
\item[(2)] Suppose $\mathbf{E}_i$ is contained in a connected component $FI(X_l,\mathcal{R})$ of $FI(X,\mathcal{R})$, $l\in \mathcal{J}$, then we have $$\bigcup_{x\in V(E_i)} V_x=X_l.$$
\item[(3)] If $x\simeq y$ in $\mathbf{E}_i$, then $V_x=\{x\}$ and $V_y=\{y\}$.
\end{itemize}
\end{lemma}
\begin{proof}
(1) Suppose $V_x\cap V_y\neq \emptyset$, then there exists an element $z\in V_x\cap V_y$. So there are $ {x_1} ,{x_2}, \cdots ,{x_m} \notin V({\bf E}_i)$ 
such that $x \sim {x_1},{x_1} \sim {x_2}, \cdots ,{x_{m - 1}} \sim {x_m},{x_m} \sim z$. Similarly, 
there are $ {y_1} ,{y_2}, \cdots ,{y_k} \notin V(\mathbf{E}_i)$ such that $y \sim {y_1},{y_1} \sim {y_2}, \cdots ,{y_{k - 1}} \sim {y_k},{y_k} \sim z$. 
Without loss of generality, we can assume that $\{x_1,\cdots, x_m\}\cap \{y_1,\cdots, y_k\}=\emptyset$. 
Otherwise, there exists a minimal number $j$ in $\{1,2,\cdots, m\}$ such that $x_j\in \{x_1,\cdots, x_m\}\cap \{y_1,\cdots, y_k\}$. 
And hence $x_j\in  V_x\cap V_y$. In this case, we can replace $z$ by $x_j$.
On the other hand, since $x,y \in V(\mathbf{E}_i)$, there are $u_1,u_2,\cdots u_s \in V(\mathbf{E}_i)$ such 
that $x\sim u_1,u_1\sim u_2,\cdots, u_s\sim y$. Thus we obtain a cycle 
$x\sim u_1,u_1\sim u_2,\cdots, u_s\sim y, y \sim {y_1},{y_1} \sim {y_2}, \cdots ,{y_{k - 1}} \sim {y_k},{y_k} \sim z, z\sim x_m,\cdots, x_1\sim x$. 
This implies that $e_{x_{j-1}x_{j}}$ (or $e_{x_{j}x_{j-1}}$ if  $x_{j}<x_{j-1}$) is 
contained in $\mathbf{E}_i$, where $1\leq j\leq m+1$. Here $x_0$ and $x_{m+1}$ are considered as $x$ and $y$, respectively. 
But this is contradictory to the fact that $ {x_1} ,{x_2}, \cdots ,{x_m}, y \notin V(\mathbf{E}_i)$.  So we claim that $V_x\cap V_y\neq \emptyset$.

(2) Since $X_l$ is connected, then for each element $y\in X_l\setminus V(\mathbf{E}_i)$, there are 
$ {x_1} ,{x_2}, \cdots ,{x_m} \notin V(\mathbf{E}_i)$ and some $x\in V(E_i)$ such that 
$x \sim {x_1},{x_1} \sim {x_2}, \cdots ,{x_{m - 1}} \sim {x_m},{x_m} \sim y$.
That implies that for each $y\in X_l\setminus V(\mathbf{E}_i)$, there exists $x\in V(\mathbf{E}_i)$ 
such that $y\in V_x$. Hence we conclude that $\bigcup\limits_{x\in V(\mathbf{E}_i)} V_x=X_l.$

(3) If there exists $z\in V_x\setminus \{x\}$, then we see that $z\in V_y$, which is due to the fact $x\simeq y$. 
This is contradictory to the fact $V_x\cap V_y= \emptyset$. So we assert that $V_x=\{x\}$. 
In an analogous manner, one can show that $V_y=\{y\}$.
\end{proof}

Now we are in a position to state the main theorem of this section.

\begin{theorem}
Every nonlinear Lie $n$-derivation of $FI(X,\mathcal{R})$ is proper if and only if 
each connected component in $FI(X,\mathcal{R})$ contains at most one $\mathbf{E}_i$ for some $i\in \mathcal{I}$. 
This is equivalent to saying that $e_{xy}\approx e_{uv}$ if $x<y$ and $u<v$ are both contained in the same connected component of $X$.
\end{theorem}

\begin{proof}
Firstly, we assume that each connected component of $FI(X,\mathcal{R})$ contains at most one $\mathbf{E}_i$ for some $i\in \mathcal{I}$. 
By Theorem \ref{t3.1} it follows that each nonlinear Lie $n$-derivation $L$ of $FI(X,\mathcal{R})$ is the 
sum of an inner-like derivation, a transitive induced derivation and a quasi-additive induced Lie $n$-derivation. 
Note that inner-like derivations and transitive induced derivations are naturally additive derivations. Then 
$L$ is proper if and only the corresponding quasi-additive induced Lie $n$-derivation is proper. 
Since each connected component in $FI(X,\mathcal{R})$ contains at most one $\mathbf{E}_i$ for some 
$i\in \mathcal{I}$, by Proposition \ref{p4.1} we say that every quasi-additive induced Lie 
$n$-derivation of $FI(X,\mathcal{R})$ is proper. Hence every nonlinear Lie $n$-derivation of 
$FI(X,\mathcal{R})$ is proper.

Conversely, assume that every nonlinear Lie $n$-derivation of $FI(X,\mathcal{R})$ has proper form. 
We need to prove that each connected component in $FI(X,\mathcal{R})$ contains at most one $\mathbf{E}_i$ for some $i\in \mathcal{I}$.
Suppose that there is a connected component $FI(X_l,\mathcal{R})$ of $FI(X,\mathcal{R})$ which contains 
two or more equivalent classes in $\{\mathbf{E}_i\mid i\in \mathcal{I}\}$. Then we can construct a nonlinear Lie 
$n$-derivation of $FI(X, \mathcal{R})$ which is not proper.

Let $\mathbf{E}_i$ be an equivalent class which is contained in $FI(X_l,\mathcal{R})$, and 
$t\in V(\mathbf{E}_i)$ be a fixed element. Let ${\bf F}:=\{\mathbf{f}_k\mid k\in \mathcal{I}\}$ be a family of additive 
derivations on $\mathcal{R}$ such that $\mathbf{f}_i\neq 0$ and $\mathbf{f}_j=0$ for all $j\neq i$, $j\in \mathcal{I}$. 
Then we can define a map $\psi:FI(X,\mathcal{R})\to FI(X,\mathcal{R})$ as follows:
 \begin{eqnarray}
{\psi}\left({\sum\limits_{x \le y} {r_{xy}{e_{xy}}} } \right) = \sum\limits_{x < y} {{\psi _F}({r_{xy}}{e_{xy}})}+ \psi(\sum\limits_{x\in X} r_{xx}e_{xx}), \quad{r_{xy}} \in \mathcal{R}, \label{eq4.10}\\
{\psi}(r{e_{xy}}) = {\mathbf{f}_k}(r){e_{xy}},\quad{e_{xy}} \in {{\bf{E}}_k},\ k\in \mathcal{J},\ r \in \mathcal{R},\label{eq4.11}\\
\psi(\sum\limits_{x\in X} r_{xx}e_{xx}) = \sum\limits_{x\in V({\bf E}_i)} \sum\limits_{y\in V_x} (\mathbf{f}_i(r_{xx})-\mathbf{f}_i(r_{tt}))e_{yy}.\label{eq4.12}
\end{eqnarray}
Let us next show that $\psi$ is a nonlinear Lie derivation of $FI(X, \mathcal{R})$. And hence it is naturally a 
nonlinear Lie $n$-derivation of $FI(X, \mathcal{R})$. Taking into account the relation \eqref{eq4.10}, we need to prove that
\begin{eqnarray}
\psi([re_{xy}, se_{uv}])=[\psi(re_{xy}), se_{uv}]+[re_{xy}, \psi(se_{uv})],\label{eq4.13}
\end{eqnarray}
and 
\begin{eqnarray}
\psi([\sum\limits_{x\in X} r_{xx}e_{xx}, se_{uv}])=[\psi(\sum\limits_{x\in X} r_{xx}e_{xx}), se_{uv}]+[\sum\limits_{x\in X} r_{xx}e_{xx}, \psi(se_{uv})],\label{eq4.14}
\end{eqnarray}
where $x<y, u<v$, and $r,s, r_{xx} \in \mathcal{R}$ for all $x\in X$.

Let us first prove \eqref{eq4.13}. If $e_{xy}\notin \mathbf{E}_i$ or $e_{uv}\notin \mathbf{E}_i$ , then it is 
not difficult to see that both sides of \eqref{eq4.13} are zero. If $[e_{xy}, e_{uv}]=0$, the both sides of \eqref{eq4.13} are zero as well. 
Now suppose that $e_{xy}, e_{uv}\in \mathbf{E}_i$ and $[e_{xy}, e_{uv}]\neq 0$. By the skew-symmetry property of Lie product, 
we only need to consider two cases: (1) $y=u$ and $x\neq v$; (2) $y=u$ and $x=v$.
For the case (1), we have 
$$
\begin{aligned}
\psi([re_{xy}, se_{uv}])&=\mathbf{f}_i(rs)e_{xv}\\
&=\mathbf{f}_i(r)s+r\mathbf{f}_i(s)e_{xv}\\
&=[\psi(re_{xy}), se_{uv}]+[re_{xy}, \psi(se_{uv})].
\end{aligned}
$$
So \eqref{eq4.13} indeed holds true under the case (1). For the case (2), we see that  $x\simeq y$. And hence $V_x=\{x\}, V_y=\{y\}$  by Lemma \ref{l4.2}. 
By invoking the relation \eqref{eq4.12},  we see that the left hand side of \eqref{eq4.13} is $$\psi([re_{xy}, se_{yx}])=\psi(rs(e_{xx}-e_{yy}))=\mathbf{f}_i(rs)e_{xx}-\mathbf{f}_i(rs)e_{yy},$$
and that the right hand side of \eqref{eq4.13} is 
\begin{equation*}
\begin{split}
[\psi(re_{xy}), se_{uv}]+[re_{xy}, \psi(se_{uv})]=&[\mathbf{f}_i(r)e_{xy}, se_{uv}]+[re_{xy}, \mathbf{f}_i(s)e_{uv}]\\
=&\mathbf{f}_i(r)s(e_{xx}-e_{yy})+r\mathbf{f}_i(s)(e_{xx}-e_{yy})\\
=&\mathbf{f}_i(rs)(e_{xx}-e_{yy}).
\end{split}
\end{equation*}
We also get \eqref{eq4.13} under the case (2).

Let us next prove the relation \eqref{eq4.14}. If $e_{uv}\notin \mathbf{E}_i$, by (2) of Lemma \ref{l4.2} we 
assert that $u,v\in V_z$ for some $z\in V(\mathbf{E}_i)$, which is due to the fact $u\sim v$. 
Suppose $e_{uv}\in \mathbf{E}_j$ for some $j\in \mathcal{I}$. Then 
\begin{eqnarray*}
\psi([\sum\limits_{x\in X} r_{xx}e_{xx}, se_{uv}])=\psi((r_{uu}-r_{vv})se_{uv})=\mathbf{f}_j((r_{uu}-r_{vv})s)e_{uv}=0.
\end{eqnarray*}
On the other hand, we have 
\begin{eqnarray*}
&&[\psi(\sum\limits_{x\in X} r_{xx}e_{xx}), se_{uv}]+[\sum\limits_{x\in X} r_{xx}e_{xx}, \psi(se_{uv})]\\
&&=[\sum\limits_{x\in V({\bf E}_i)} \sum\limits_{y\in V_x} (\mathbf{f}_i(r_{xx})-\mathbf{f}_i(r_{tt}))e_{yy}, se_{uv}]+[\sum\limits_{x\in X} r_{xx}e_{xx}, \mathbf{f}_j(s)e_{uv}]\\
&&=[\sum\limits_{y\in V_z} (\mathbf{f}_i(r_{zz})-\mathbf{f}_i(r_{tt}))e_{yy}, se_{uv}]\\
&&=0.
\end{eqnarray*}
So we obtain \eqref{eq4.14} whenever $e_{uv}\notin {\bf E}_i$. If $e_{uv}\in \mathbf{E}_i$, then the left hand side of \eqref{eq4.14} is 
\begin{eqnarray*}
\psi([\sum\limits_{x\in X} r_{xx}e_{xx}, se_{uv}])=\psi((r_{uu}-r_{vv})se_{uv})=\mathbf{f}_i((r_{uu}-r_{vv})s)e_{uv},
\end{eqnarray*}
while the right hand side of \eqref{eq4.14} is 
\begin{eqnarray*}
&&[\psi(\sum\limits_{x\in X} r_{xx}e_{xx}), se_{uv}]+[\sum\limits_{x\in X} r_{xx}e_{xx}, \psi(se_{uv})]\\
&&=[\sum\limits_{x\in V(\mathbf{E}_i)} \sum\limits_{y\in V_x} (\mathbf{f}_i(r_{xx})-\mathbf{f}_i(r_{tt}))e_{yy}, se_{uv}]+[\sum\limits_{x\in X} r_{xx}e_{xx}, \mathbf{f}_i(s)e_{uv}]\\
&&=[\mathbf{f}_i(r_{uu})-\mathbf{f}_i(r_{vv})]se_{uv}+(r_{uu}-r_{vv})\mathbf{f}_i(s)e_{uv}\\
&&=\mathbf{f}_i(r_{uu}-r_{vv})se_{uv}+(r_{uu}-r_{vv})\mathbf{f}_i(s)e_{uv}\\
&&=\mathbf{f}_i((r_{uu}-r_{vv})s)e_{uv}.
\end{eqnarray*}
Thus we get \eqref{eq4.14} whenever $e_{uv}\in {\bf E}_i$.  
We therefore conclude that $\psi$ is a nonlinear Lie derivation of $FI(X,\mathcal{R})$. 

In view of the relations \eqref{eq4.11} and \eqref{eq4.12}, we know that $\psi$ is a quasi-additive induced Lie $n$-derivation. 
Since $\mathbf{f}_i\neq 0$ and $\mathbf{f}_j=0$ for all $j\neq i$, and $FI(X_l,\mathcal{R})$ 
contains two or more equivalent classes in $\{\mathbf{E}_j\mid j\in \mathcal{I}\}$ by hypothesis,  
$\psi$ is not proper by Proposition \ref{p4.1}. This is contradictory to the fact that every nonlinear 
Lie $n$-derivation of $FI(X,\mathcal{R})$ is of proper. 

We eventually complete the proof of this theorem.
\end{proof}

\section{Related Topics for Further Research}
\label{xxsec5}

The question of to what extent the multiplicative structure (alternatively speaking, 
the structure of nonlinear maps without additivity) of an algebra determines its
additive structure has been paying considerable attention by many researchers  over the past decades.
In particular, they have investigated under which conditions bijective mappings
between algebras preserving the multiplicative structure necessarily preserve the
additive structure as well. The most fundamental result in this direction is due
to W. S. Martindale III \cite{Mar69} who  proved that every bijective
multiplicative mapping from a prime ring containing a nontrivial idempotent
onto an arbitrary ring is necessarily additive. Later, a number of authors
considered the Jordan-type product or Lie-type product and proved that,
on certain associative algebras or rings, bijective mappings which preserve
any of those products are automatically additive, see
\cite{FosnerWeiXiao, WangWang, XiaoWei, Yang2020, Yang20211, Yang20212}.

Let us come back the topics of finitary incidence algebras. Although the involutions of 
incidence algebras are not involved in the current work, Brusamarello, Lewis, Spiegel and Donnell et al
\cite{BruL, Spiegel2,  SpiegelDonnell} characterized the involutions of incidence algebras and 
got some powerful results towards these additive mappings.  This will bring us some enlightenment, that is, it 
is worthy of studying the Jordan-type and Lie-type derivations with involutions on finitary incidence algebras.

\subsection{(Non-)Linear $\ast$-Lie-type derivations}
\label{xxsec5.1}

Let $\mathcal{R}$ be a commutative ring and $\mathcal{A}$ be an associative $\mathcal{R}$-algebra 
with involution $\ast$. Let $d: \mathcal{A}\longrightarrow \mathcal{A}$ be an $\mathcal{R}$-linear mapping. 
We say that $d$ is a
\textit{$\ast$-Lie derivation} if
$$
d([x, y]_\ast)=[d(x), y]_\ast+[y, d(x)]_\ast,
$$
holds true for all $x,y\in \mathcal{A}$, where $[x, y]_\ast=xy-yx^\ast$ is the so-called $\ast$-Lie product. 
Similarly, an $\mathcal{R}$-linear mapping $d:
\mathcal{A}\longrightarrow \mathcal{A}$ is called a \textit{$\ast$-Lie triple derivation} if it satisfies the condition
$$
d([[x,  y]_\ast,  z]_\ast)=[[d(x), y]_\ast, z]_\ast+[[x, d(y)]_\ast, z]_\ast+[[x, y]_\ast,   d(z)]_\ast
$$
for all $x, y, z\in \mathcal{A}$.

Given the consideration of $\ast$-Lie derivations and $\ast$-Lie triple derivations, one can
further develop them in more general way. Suppose that $n\geq 2$ is a fixed positive
integer. Let us see a sequence of polynomials with involution $\ast$
$$
\begin{aligned}
p_1(x_1)&=x_1,\\
p_2(x_1, x_2)&=[p_1(x_1), x_2]_\ast=[x_1, x_2]_\ast,\\
p_3(x_1, x_2, x_3)&=[p_2(x_1, x_2), x_3]_\ast=[[x_1, x_2]_\ast,  x_3]_\ast\\
p_4(x_1, x_2, x_3, x_4)&=[p_3(x_1, x_2, x_3), x_4]_\ast=[[[x_1, x_2]_\ast, x_3]_\ast, x_4]_\ast\\
& \vdots \ \ \ \ \ \ \ \  \ \ \ \ \ \ \ \ \ \  \vdots \ \ , \\
p_n(x_1, x_2,\cdots, x_n)&=[p_{n-1}(x_1, x_2,\cdots, x_{n-1}), x_n]_\ast\\
&=\underbrace{[[\cdots [[}_{n-1}x_1,  x_2]_\ast, x_3]_\ast, \cdots,  x_{n-1}]_\ast, x_n]_\ast\\
\end{aligned}
$$
A \textit{$\ast$-Lie $n$-derivation} is an $\mathcal{R}$-linear mapping $d:
\mathcal{A} \longrightarrow \mathcal{A}$ satisfying the condition
$$
d(p_n(x_1, x_2,\cdots, x_n))=\sum_{k=1}^n
p_n(x_1,\cdots, x_{k-1}, d(x_k), x_{k+1},\cdots, x_n)
$$
for all $x_1,x_2,\cdots, x_n\in \mathcal{A}$. By the definition, it is clear that every $\ast$-Lie derivation is an
$\ast$-Lie 2-derivation and each $\ast$-Lie triple derivation is an $\ast$-Lie 3-derivation.
One can easily check that each nonlinear $\ast$-Lie derivation on $\mathcal{A}$ is
a nonlinear $\ast$-Lie triple derivation. But, we don't know whether the converse statement is true.
$\ast$-Lie 2-derivations, $\ast$-Lie 3-derivations and $\ast$-Lie $n$-derivations
are collectively referred to as \textit{$\ast$-Lie-type derivations}. Those $\ast$-Lie-type derivations 
without the additivity assumption are called \textit{nonlinear $\ast$-Lie-type derivations}.
(Non-)Linear $\ast$-Lie-type derivations on various algebras are intensively studied 
by some researchers,  see \cite{Lin1, Lin2, Lin3}.

\begin{question}\label{Ques5.1}
Let $\mathcal{R}$ be a 2-torsionfree and $(n-1)$-torsionfree commutative ring with identity, 
and $FI(X, \mathcal{R})$ be the finitary incidence algebra over a pre-ordered set $(X, \leq)$.   
What can we say about (non-)linear $\ast$-Lie-type derivations of $FI(X, \mathcal{R})$ ? 
Is every (non-)linear $\ast$-Lie-type derivation on $FI(X, \mathcal{R})$ an additive $\ast$-Lie derivation ? 
Can we fond a reasonable condition which enables every (non-)linear $\ast$-Lie-type derivation on $FI(X, \mathcal{R}) $ to be an 
additive $\ast$-derivation ? 
\end{question}

\subsection{(Non-)Linear $\ast$-Jordan-type derivations}
\label{xxsec5.2}

Let $\mathcal{R}$ be a commutative ring and $\mathcal{A}$ be an associative $\mathcal{R}$-algebra 
with involution $\ast$. Let $d: \mathcal{A}\longrightarrow \mathcal{A}$ be a mapping
(without the additivity assumption). We say that $d$ is a
\textit{nonlinear $\ast$-Jordan derivation} if
$$
d(x \circ_\ast y)=d(x) \circ_\ast y+y \circ_\ast d(x),
$$
holds true for all $x,y\in \mathcal{A}$, where $x\circ_\ast y=xy+yx^\ast$ is the so-called $\ast$-Jordan product. 
Similarly, a mapping $d:
\mathcal{A}\longrightarrow \mathcal{A}$ is called a \textit{nonlinear $\ast$-Jordan triple derivation} if it satisfies the condition
$$
d(x \circ_\ast y \circ_\ast z)=d(x) \circ_\ast y \circ_\ast z+x \circ_\ast d(y)\circ_\ast z+x \circ_\ast y \circ_\ast d(z)
$$
for all $x, y, z\in \mathcal{A}$, where $x \circ_\ast y \circ_\ast  z=(x \circ_\ast y) \circ_\ast z$ .
We should be aware that the $\ast$-Jordan product $\circ_\ast$ is not necessarily associative.

Given the consideration of $\ast$-Jordan derivations and $\ast$-Jordan triple derivations, one can
further develop them in more general way. Suppose that $n\geq 2$ is a fixed positive
integer. Let us see a sequence of polynomials with involution $\ast$
$$
\begin{aligned}
p_1(x_1)&=x_1,\\
p_2(x_1, x_2)&=p_1(x_1)\circ_\ast x_2=x_1\circ_\ast x_2,\\
p_3(x_1, x_2, x_3)&=p_2(x_1, x_2) \circ_\ast x_3=(x_1\circ_\ast x_2) \circ_\ast x_3\\
p_4(x_1, x_2, x_3, x_4)&=p_3(x_1, x_2, x_3) \circ_\ast x_4=((x_1 \circ_\ast x_2) \circ_\ast x_3) \circ_\ast x_4,\\
& \vdots \ \ \ \ \ \ \ \  \ \ \ \ \ \ \ \ \ \  \vdots \ \ , \\
p_n(x_1, x_2,\cdots, x_n)&=p_{n-1}(x_1, x_2,\cdots, x_{n-1}) \circ_\ast x_n\\
&=\underbrace{(\cdots ((}_{n-2}x_1 \circ_\ast x_2)\circ_\ast x_3) \cdots \circ_\ast x_{n-1}) \circ_\ast x_n\\
\end{aligned}
$$
A \textit{nonlinear $\ast$-Jordan $n$-derivation} is a mapping $d:
\mathcal{A} \longrightarrow \mathcal{A}$ satisfying the condition
$$
d(p_n(x_1, x_2,\cdots, x_n))=\sum_{k=1}^n
p_n(x_1,\cdots, x_{k-1}, d(x_k), x_{k+1},\cdots, x_n)
$$
for all $x_1,x_2,\cdots, x_n\in \mathcal{A}$. This notion makes the best use of the definition of $\ast$-Jordan-type derivation
and that of $\ast$-Lie-type derivation,  see \cite{FosnerWeiXiao, Jing, Lin1, Lin2, Lin3}.
By the definition, it is clear that every $\ast$-Jordan derivation is an
$\ast$-Jordan 2-derivation and each $\ast$-Jordan triple derivation is an $\ast$-Jordan 3-derivation.
One can easily check that each nonlinear $\ast$-Jordan derivation on $\mathcal{A}$ is
a nonlinear $\ast$-Jordan triple derivation. But, we don't know whether the converse statement is true.
$\ast$-Jordan 2-derivations, $\ast$-Jordan 3-derivations and $\ast$-Jordan $n$-derivations
are collectively referred to as \textit{$\ast$-Jordan-type derivations}.

\begin{question}\label{Ques5.2}
Let $\mathcal{R}$ be a 2-torsionfree and $(n-1)$-torsionfree commutative ring with identity, 
and $FI(X, \mathcal{R})$ be the finitary incidence algebra over a pre-ordered set $(X, \leq)$.   
What can we say about (non-)linear $\ast$-Jordan-type derivations of $FI(X, \mathcal{R})$ ? 
Is every (non-)linear $\ast$-Jordan-type derivation on $FI(X, \mathcal{R})$ an additive $\ast$-derivation ? 
\end{question}

Although the main purpose of the current article is to study
nonlinear Lie-type derivations of finitary incidence algebras, the structure of biderivations
of finitary incidence algebras also has a great interest and much more attention should be paid on them.
In addition, we extend the definition of Lie biderivation to the case of arbitrary Lie-type
biderivation in a general way. Given the parallels pointed out in many literatures
between Lie structure and Jordan structure, we propose the definition of Jordan-type biderivation.
It is natural to ask whether properties known for Lie-type biderivations might also be valid for Jordan-type biderivations. We give several sample
questions of this sort.  In this section, we will present several
potential topics sharing common interest for our future research in this line.

\subsection{(Non-)Linear biderivations}
\label{xxsec5.3}

Let $\mathcal{R}$ be a commutative ring with identity and $\mathcal{A}$ be an associative $\mathcal{R}$-algebra 
with center $\mathcal{Z(A)}$.
An $\mathcal{R}$-bilinear mapping $\varphi: \mathcal{A}\times \mathcal{A}\longrightarrow \mathcal{A}$
is a \textit{biderivation} if it is a derivation
with respect to both components, that is
$$
\varphi(xz, y)=\varphi(x,y)z+x\varphi(z,y)~\text{and}~\varphi(x,yz)=\varphi(x,y)z+y\varphi(x,z)
$$
for all $x, y, z\in \mathcal{A}$. A biderivation which 
does not necessarily satisfy the bilinear condition is called a \textit{nonlinear biderivation}. 
If the algebra $\mathcal{A}$ is noncommutative,
then the mapping $\varphi(x, y)=\lambda[x, y]$ for all $x,y\in \mathcal{A}$ and
some $\lambda\in \mathcal{Z(A)}$ is called an \textit{inner biderivation}.
An $\mathcal{R}$-bilibear mapping $\varphi: \mathcal{A}\times \mathcal{A}\longrightarrow \mathcal{A}$
is said to be an \textit{extremal biderivation}
if it is of the form $\varphi(x,y)=[x,[y, a]]$ for all $x, y\in \mathcal{A}$ and some $a\in \mathcal{A}, a\notin \mathcal{Z(A)}$.
A map $f: \mathcal{A}\longrightarrow \mathcal{A}$ is called \textit{centralizing} on $\mathcal{A}$ if $[f(x), x]\in \mathcal{Z(A)}$ 
for all $x\in \mathcal{A}$. In the special case when $[f(x), x]=0$ for all $x\in \mathcal{A}$, $f$ is called \textit{commuting} on $\mathcal{A}$. 
The notion of additive commuting maps is closely related to the notion of biderivations. In fact, every commuting 
additive mapping $f$ on $\mathcal{A}$ gives raise to a biderivation of $\mathcal{A}$: linearizing $[f(x), x]=0$, we 
observe that $[f(x), y]=[x, f(y)]$ for all $x, y\in \mathcal{A}$, so the mapping $(x, y) \longmapsto [f(x), y]$
is easily seen to be a biderivation of $\mathcal{A}$ that is inner in each component (induced by $f(y)$ and $f(x)$, 
respectively). A biderivation $\varphi$ of $\mathcal{A}$ is called an \textit{inner biderivation} if there exists 
$\lambda\in \mathcal{Z(A)}$  such that $\varphi(x, y)=\lambda [x, y]$ for all $x, y\in \mathcal{A}$. 
Inner biderivations appear quite naturally in characterizing biderivations of associative algebras.

Let $\mathcal{C}$ be a commutative domain with identity. 
It is well-known that each biderivation of block upper triangular algebra $\mathcal{B}_n^{\bar{k}}(\mathcal{C})$ 
the sum of an extremal biderivation and an inner biderivation. In particular, every biderivation of $\mathcal{T}_n(\mathcal{C})$ 
is the sum of an extremal biderivation and an inner biderivation, see \cite{Benkovic}. Let $\mathcal{A}$ be unital algebra and 
$\mathcal{M}_n(\mathcal{A})$ be the full matrix algebra over $\mathcal{A}$. Du and Wang \cite{DuWang} proved that 
each biderivation on $\mathcal{M}_n(\mathcal{A})$ is inner. To study Poisson structure of incidence algebras, 
Kaygorodov and Khrypchenko describe detailedly antisymmetric biderivations of incidence algebras \cite{KK}.
 In view of the relationships between matrix algebras and 
incidence algebras, it is natural to consider biderivations of finitary incidence algebras.

\begin{question}\label{Ques5.3}
Let $\mathcal{R}$ be a 2-torsionfree and $({ n}-1)$-torsionfree commutative ring with identity, 
and $FI(X, \mathcal{R})$ be the finitary incidence algebra over a pre-ordered set $(X, \leq)$.   
What can we say about (non-)linear  biderivations of $FI(X, \mathcal{R})$ ? 
or, every (non-)linear biderivation on  $FI(X, \mathcal{R})$ is the sum of an extremal 
biderivation and an inner biderivation. 
\end{question}

For an arbitrary associative algebra $\mathcal{A}$, if each biderivation of $\mathcal{A}$ is inner, 
then every commuting map of $\mathcal{A}$ is of the so-called standard form or proper form. However, 
Jia and Xiao  \cite{JiaXiao} observed that the commuting maps of incidence algebras are not necessarily 
proper form. We therefore speculate that certain biderivations on $\mathcal{I}(X, \mathcal{R})$ are not inner.

\subsection{(Non-)Linear Lie-type biderivations}
\label{xxsec5.4}

Let $\mathcal{R}$ be a commutative ring with identity and $\mathcal{A}$ be an associative $\mathcal{R}$-algebra with center $\mathcal{Z(A)}$.
Mimicking the definitions of Lie biderivations in our sense and Lie-type derivations in the sense of $(\spadesuit)$, one can give the definition of
arbitrary Lie-type biderivations on $\mathcal{A}$.
An $\mathcal{R}$-bilinear mapping $\varphi:
\mathcal{A}\times \mathcal{A}\longrightarrow \mathcal{A}$ is called a
\textit{Lie $n$-biderivation} if it is a Lie $n$-derivation
with respect to both components, implying that
$$
\varphi(p_n(x_1,x_2,\dots,x_n), y)=\sum_{i=1}^n p_n(x_1,\dots,x_{i-1}, \varphi(x_i, y),x_{i+1},\dots,x_n)
$$
and
$$
\varphi(x, p_n(y_1, y_2,\dots, y_n))=\sum_{i=1}^n p_n(y_1,\dots, y_{i-1}, \varphi(x, y_i), y_{i+1},\dots, y_n)
$$
holds for all $x_1,x_2,\dots,x_n\in \mathcal{A}$.
 It is clear that
biderivations and inner biderivations are specific examples of Jordan biderivations .
By definition, a Lie biderivation is a Lie $2$-biderivation and a Lie triple biderivation is a Lie
$3$-biderivation. It is straightforward to check that every Lie $n$-biderivation
on $\mathcal{A}$ is a Lie $(n+k(n-1))$-biderivation for all $k\in \Bbb{N}_0$.
Lie $2$-biderivations, Lie $3$-biderivations and Lie $n$-biderivations are
collectively referred to as \textit{Lie-type biderivations}. Those Lie-type biderivations which 
do not necessarily satisfy the bilinear condition are said to be \textit{nonlinear Lie-type biderivations}. 
Let us see the concrete construction on
Lie triple biderivations.  A mapping $\varphi: \mathcal{A}\times \mathcal{A}\longrightarrow \mathcal{A}$
is a \textit{nonlinear Lie triple biderivation} if it is a Lie triple derivation
with respect to both components, implying that
$$
\varphi([w, [x,z]],y)=[\varphi(w, y), [x, z]]+[w, [\varphi(x,y), z]]+[w, [x,\varphi(z,y)]]
$$
and
$$
\varphi(x,[[y,z], w])=[[\varphi(x,y), z], w]+[[y,\varphi(x,z)], w]+[[y, z], \varphi(x, w)]
$$
for all $x,y\in \mathcal{A}$.

\begin{question}\label{Ques5.4}
Let $\mathcal{R}$ be a 2-torsionfree and $({n}-1)$-torsionfree commutative ring with identity, 
and $FI(X, \mathcal{R})$ be the finitary incidence algebra over a pre-ordered set $(X, \leq)$.   
What can we say about (non-)linear Lie-type biderivations of $FI(X, \mathcal{R})$ ? 
We want to know what kind of decomposition formula can be obtained for $FI(X, \mathcal{R})$ ?
\end{question}

\subsection{(Non-)Linear Jordan-type biderivations}\label{xxsec5.5}
Let $\mathcal{R}$ be a commutative ring with identity and $\mathcal{A}$ be an associative $\mathcal{R}$-algebra with center $\mathcal{Z(A)}$.
Denote by $x\circ y=xy+yx$ the Jordan product of elements $x, y\in \mathcal{A}$.
An $\mathcal{R}$-linear (or an additive) map $d: \mathcal{A}\longrightarrow \mathcal{A}$ is called
a \textit{Jordan derivation} if $d(x\circ y)=d(x)\circ y+x\circ d(y)$ holds for all $x, y\in \mathcal{A}$,
or equivalently, if $d(x^2)=d(x)x+xd(x)$ holds for all $x\in \mathcal{A}$ in the case that the characteristic
of $\mathcal{R}$ is not 2; is called a \textit{Jordan triple derivation} if
$d(x\circ y \circ z)=d(x)\circ y \circ z+x\circ d(y) \circ z+x\circ y\circ d(z)$
 holds for all $x, y, z\in \mathcal{A}$, where $x\circ y \circ z=(x \circ y) \circ z$.

 Given the consideration of Jordan derivations and Jordan triple derivations, we can
further develop them in one natural way. Suppose that $n\geq 2$ is a fixed positive
integer. Let us see a sequence of polynomials
$$
\begin{aligned}
q_1(x_1)&=x_1,\\
q_2(x_1,x_2)&=x_1 \circ x_2=x_1x_2+x_2x_1,\\
q_3(x_1,x_2,x_3)&=(q_2(x_1,x_2)) \circ x_3=(x_1 \circ x_2) \circ x_3,\\
q_4(x_1,x_2,x_3,x_4)&=(q_3(x_1,x_2,x_3)) \circ x_4=((x_1 \circ x_2) \circ x_3) \circ x_4,\\
\cdots &\cdots,\\
q_n(x_1,x_2,\cdots,x_n)&=(q_{n-1}(x_1,x_2,\cdots,x_{n-1})) \circ x_n\\
&=\underbrace{(\cdots ((}_{n-2}x_1 \circ x_2)\circ x_3) \cdots \circ x_{n-1}) \circ x_n.
\end{aligned}
$$
The polynomial $q_n(x_1,x_2,\cdots,x_n)=(q_{n-1}(x_1, x_2,\cdots x_{n-1}))\circ x_n$
for $n\geq 2$, which is called the \textit{Jordan $n$-product} of $x_1, x_2, \cdots, x_n$.
Accordingly, a \textit{Jordan $n$-derivation} is an $\mathcal{R}$-linear mapping $\varphi:
\mathcal{A} \longrightarrow \mathcal{A}$ satisfying the condition
$$
\delta(q_n(x_1, x_2,\cdots, x_n))=\sum_{i=1}^n
q_n(x_1,\cdots, x_{i-1}, \delta(x_k), x_{i+1},\cdots, x_n)
$$
for all $x_1, x_2,\cdots, x_n\in \mathcal{A}$. This notion makes the best use of the definition of Lie-type derivations
and that of $\ast$-Lie-type derivations,  see \cite{Abdullaev, BenkovicEremita, FosnerWeiXiao}.
By the definition, it is clear that every Jordan derivation is a
Jordan 2-derivation and each Jordan triple derivation is a Jordan 3-derivation.
Jordan $2$-derivations, Jordan $3$-derivations and Jordan $n$-derivations are
collectively referred to as \textit{Jordan-type derivations}.
One can easily check that each Jordan derivation on $\mathcal{A}$ is
a Jordan triple derivation. But, we don't know whether the converse statement is true.
An important formula $[[x, y], z]=x\circ (y\circ z)-y \circ (x\circ z)$, shows that every Jordan derivation is also a Lie triple derivation.
Therefore, studying Lie triple derivations enables us to treat both important
class of Jordan derivations and Lie derivations simultaneously.
More recently, Jordan-type derivations on triangular algebras, prime rings, matrix algebras,
nest algebras and von Neumann algebras are considered by Lin,  Qi and Zhao et al, see
\cite{Lin3, YuZhang1, ZhaoLi1}.

An $\mathcal{R}$-bilinear mapping $\varphi: \mathcal{A}\times \mathcal{A}\longrightarrow \mathcal{A}$
is a \textit{Jordan biderivation} if it is a Jordan derivation
with respect to both components, implying that
$$
\varphi(x\circ y, z)=\varphi(x, z)\circ y+x\circ \varphi(y, z)~\text{and}~\varphi(x, y\circ z)=\varphi(x,y)\circ z+y\circ \varphi(x, z)
$$
for all $x,y\in \mathcal{A}$.  Generally speaking, an $\mathcal{R}$-bilinear mapping $\varphi:
\mathcal{A}\times \mathcal{A}\longrightarrow \mathcal{A}$ is called a
\textit{Jordan $n$-biderivation} if it is a Jordan $n$-derivation
with respect to both components, implying that
$$
\varphi(q_n(x_1,x_2,\dots,x_n), y)=\sum_{i=1}^n q_n(x_1,\dots,x_{i-1}, \varphi(x_i, y),x_{i+1},\dots,x_n)
$$
and
$$
\varphi(x, q_n(y_1, y_2,\dots, y_n))=\sum_{i=1}^n q_n(y_1,\dots, y_{i-1}, \varphi(x, y_i), y_{i+1},\dots, y_n)
$$
holds for all $x_1,x_2,\dots,x_n\in \mathcal{A}$.
It is clear that
biderivations and inner biderivations are specific examples of Jordan biderivations.
By definition, a Jordan biderivation is a Jordan 2-biderivation and a Jordan triple biderivation is a Jordan
3-biderivation. Jordan 2-biderivations, Jordan 3-biderivations and Jordan $n$-biderivations are
collectively referred to as \textit{Jordan-type biderivations}.  It is straightforward to check that every Jordan bibiderivation
on $\mathcal{A}$ is a Lie triple biderivation.
Those Jordan-type biderivations which 
do not necessarily satisfy the bilinear condition are said to be \textit{nonlinear Jordan-type biderivations}. 
Let us see the intuitional and concrete construction on
Lie triple biderivations.  A mapping $\varphi: \mathcal{A}\times \mathcal{A}\longrightarrow \mathcal{A}$
is a \textit{nonlinear Jordan triple biderivation} if it is a Lie triple derivation
with respect to both components, implying that
$$
\varphi(w\circ x \circ z, y)=\varphi(w, y)\circ x\circ z+w\circ \varphi(x,y)\circ z+w\circ x\circ \varphi(z,y)
$$
and
$$
\varphi(x, y\circ z\circ w)=\varphi(x,y)\circ z\circ w+y\circ\varphi(x,z)\circ w+y\circ z\circ \varphi(x, w)
$$
for all $x,y\in \mathcal{A}$.

\vspace{2mm}

By \cite{Xiao} and \cite{Kh-Jor} we can say that 

\begin{proposition}\label{Prop5.5}
Let $\mathcal{R}$ be a 2-torsionfree commutative ring with identity, $(X, \leq)$ be a locally finite pre-ordered set, 
and $FI(X, \mathcal{R})$ be the finitary incidence algebra over $(X, \leq)$.   Then each 
bilinear Jordan biderivation on a incidence algebra $FI(X, \mathcal{R})$ is a biderivation. 
\end{proposition}

\begin{question}\label{Ques5.6}
Let $\mathcal{R}$ be a 2-torsionfree and $({n}-1)$-torsionfree commutative ring with identity, 
and $FI(X, \mathcal{R})$ be the finitary incidence algebra over a pre-ordered set $(X, \leq)$.   
What can we say about (non-)linear Jordan-type derivations of $FI(X, \mathcal{R})$ ? 
or, every (non-)linear Jordan-type biderivation on $FI(X, \mathcal{R})$ is an additive 
biderivation ? 

\end{question}

\addtocontents{toc}{
    \protect\settowidth{\protect\@tocsectionnumwidth}{}%
    \protect\addtolength{\protect\@tocsectionnumwidth}{0em}}


\begin{thebibliography}{}


 \bibitem[1]{Abdullaev}
I. Z. Abdullaev,
{\em $n$-Lie derivations on von
Neumann algebras}, Uzbek. Mat. Zh., \textbf{5-6} (1992), 3-9.



 \bibitem[2]{Baclawski} 
K. Baclawski, 
{\em Automorphisms and derivations of incidence algebras}, 
Proc. Amer. Math. Soc., \textbf{36} (1972), 351-356.



\bibitem[3]{Benkovic} 
D. Benkovi\v c, 
{\em Biderivations on triangular matrices}, 
Linear Algebra Appl., \textbf{431} (2009), 1587-1602.







\bibitem[4]{BenkovicEremita}
D. Benkovi\v c and D. Eremita,
{\em Multiplicative Lie $n$-derivations of triangular rings}, Linear Algebra Appl., 
\textbf{436} (2012), 4223-4240. 


 
 \bibitem[5]{Bre93} 
M. Bre\v sar, 
{\it Commuting traces of biadditive maps, commutativity-preserving
maps and Lie maps}, Trans. Amer. Math Soc. {\bf 335} (1993), 525-546.



\bibitem[6]{BruFK1} 
R. Brusamarello, E. Z. Fornaroli and M. Khrypchenko,
{\it Jordan automorphisms of finitary incidence algebras},
Linear Multilinear Algebra, {\bf 66} (2018), 565-579.




\bibitem[7]{BruFK2} 
R. Brusamarello, E. Z. Fornaroli and M. Khrypchenko,
{\it Jordan isomorphisms of the finitary incidence rings of a partially ordered category},
Colloq. Math., {\bf 159} (2020), 285-307.




\bibitem[8]{BruFS3} 
R. Brusamarello, E. Z. Fornaroli and E. A. Santulo Jr,
{\it Multiplicative automorphisms of incidence algebras},
Comm. Algebra, {\bf 43} (2015), 726-736.



\bibitem[9]{BruL} 
R. Brusamarello and D. Lewis, 
{\it Antomorphisms and involutions
on incidence algebras}, Linear Multilinear Algebra, {\bf 59} (2011), 1247-1267.



\bibitem[10]{ChenZhang} 
L. Chen and J.-H. Zhang, 
{\em Nonlinear Lie derivations on upper triangular matrices}. 
Linear Multilinear Algebra, {\bf56}(2008), 725-730.



\bibitem[11]{CourtemancheDugasHerden}
J Courtemanche, M. Dugas and D. Herden,
{Local automorphisms of finitary incidence algebras}, 
Linear Algebra Appl., \textbf{541} (2018), 221-257.



\bibitem[12]{DuWang} 
Y.-Q. Du and Y. Wang, 
{\em Biderivations of generalized matrix algebras}, 
Linear Algebra Appl., \textbf{438} (2013), 4483-4499.



\bibitem[13]{Dugas}
M. Dugas,
{Homomorphisms of finitary incidence algebras}, 
Comm. Algebra, \textbf{40} (2012), 2373-2384.



\bibitem[14]{DugasWagner}
M. Dugas and B. Wagner,
{\em Finitary incidence algebras and idealizations}, 
Linear Multilinear Algebra, \textbf{64} (2016), 1936-1951.



\bibitem[15]{DugasHerdenRebrovich1}
M. Dugas, D. Herden and J. Rebrocivh,
{\em Normal subgroups of the group of units of incidence algebras}, 
Linear Algebra Appl.,  \textbf{586} (2020), 64-88.



\bibitem[16]{DugasHerdenRebrovich2}
M. Dugas, D. Herden and J. Rebrocivh,
{\em Indecomposable ideals of finitary incidence algebras}, 
J. Pure Appl. Algebra, \textbf{224} (2020), no. 8, 106336.




\bibitem[17]{FosnerWeiXiao}  
A. Fo\v sner, F. Wei and Z.-K. Xiao,
{\em Nonlinear Lie-type derivations of von Neumann algebras and related topics},
Colloq. Math., \textbf{132} (2013), 53-71.



\bibitem[18]{Her} 
I. N. Herstein, 
{\it Lie and Jordan structures in simple associative rings},
Bull. Amer. Math. Soc.,  {\bf 67} (1961), 517-531.



\bibitem[19]{JiLiuZhao} 
P.-S. Ji, R.-R. Liu and Y.-Z. Zhao, 
{\em Nonlinear Lie triple derivations of triangular algebras},
Linear Multilinear Algebra,  {\bf 60} (2012), 1155-1164.


\bibitem[20]{JiaXiao} 
H.-Y. Jia and Z.-K. Xiao, 
{\em Commuting maps on certain incidence algebras}, 
Bull. Iranian Math. Soc., \textbf{46} (2020), 755-765.



\bibitem[21]{Jing} 
W. Jing, 
{\em  Nonlinear $\ast$-Lie derivations of standard operator algebras}, 
Quaest. Math., \textbf{39} (2016), 1037-1046.




\bibitem[22]{KK} 
I. Kaygorodov and M. Khrypchenko, 
{\em Poisson structures on finitary incidence algebras}, J. Algebra, 
\textbf{578} (2021), 402-420. 




\bibitem[23]{KKW} 
I. Kaygorodov, M. Khrypchenko, and F. Wei,  
{\em Higher derivations of finitary incidence algebras}, Algebra and Represent Theory, 
\textbf{22} (2019), 1331-1341. 



\bibitem[24]{Kh-aut} 
N. S. Khripchenko, 
{\em Automorphisms of finitary incidence rings}, Algebra Discrete Math., \textbf{9} (2010), 78-97.



\bibitem[25]{Kh-der} 
N. S. Khripchenko, 
{\em Derivations of finitary incidence rings}, Comm. Algebra, \textbf{40} (2012), 2503-2522.




\bibitem[26]{Kh-Jor} 
M. Khrypchenko, 
{\em  Jordan derivations of finitary incidence rings}, Linear Multilinear Algebra, 
\textbf{64} (2016), 2104-2118.



\bibitem[27]{Kh-loc} 
M. Khrypchenko, 
{\em Local derivations of finitary incidence rings}, Acta Math. Hungar., 
\textbf{154} (2018), 48-55. 



\bibitem[28]{KN} 
N.S. Khripchenko and B.V. Novikov, 
{\em Finitary incidence algebras}, Comm. Algebra \textbf{37} (2009), 1670-1676.




\bibitem[29]{KhrypchenkoWei} 
M. Khrypchenko and F. Wei, 
{\em Lie-type derivations of finitary incidence algebras}, 
Rocky Mountain J. Math., \textbf{50} (2020), 101-113. 




\bibitem[30]{Kopp} 
M. Koppinen, 
{\it Automorphisms and higher derivations of incidence algebras},
J. Algebra, {\bf 174} (1995), 698-723.




\bibitem[31]{Lin1} W.-H. Lin,  
{\em Nonlinear $\ast$-Lie-type derivations on standard operator algebras}, 
Acta Math. Hungar., \textbf{154} (2018), 480-500.




\bibitem[32]{Lin2} W.-H. Lin,  
{\em Nonlinear $\ast$-Lie-type derivations on von Neumann algebras}, 
Acta Math. Hungar., \textbf{156} (2018), 112-131. 




\bibitem[33]{Lin3} W.-H. Lin,  
{\em Nonlinear $\ast$-Jordan-type derivations on von Neumann algebras}, 
\href{https://arxiv.orb/abs/1805.02037}{arXiv: 1805.02037. [math.OA]}.



\bibitem[34]{Mar64} 
W. S. Martindale III, 
{\it Lie derivations of primitive rings},
Michigan Math. J., {\bf 11} (1964), 183-187.



\bibitem[35]{Mar69} 
W. S. Martindale III, 
{\it  When are multiplicative mappings additive ?},  Proc Amer Math Soc.
\textbf{21} (1969), 695-698.



\bibitem[36]{Qi1}
X.-F. Qi,
{\em Characterizing Lie  $n$-derivations for reflexive algebras},
Linear Multilinear Algebra, \textbf{63} (2015), 1693-1706.



\bibitem[37]{Qi2}
X.-F. Qi,
{\em Lie $n$-derivations on $\mathcal{J}$-subspace lattice algebras}, 
Proc. Indian Acad. Sci. Math. Sci., \textbf{127} (2017), 537-545. 




\bibitem[38]{Spiegel1} 
E. Spiegel, 
{\it On the automorphisms of incidence algebras},
J. Algebra, {\bf 239} (2001), 615-623.




\bibitem[39]{Spiegel2} 
E. Spiegel, 
{\it Involutions in incidence algebras},
Linear Algebra Appl., , {\bf 405} (2005), 155-162.



\bibitem[40]{SpiegelDonnell} 
E. Spiegel and C. O'Donnell, 
{\it Incidence Algebras},
Monographs and Textbooks in Pure and Applied Mathematics, vol. \textbf{206}, Marcel Dekker,
New York, 1997.



\bibitem[41]{Stanley1} 
R. Stanley, 
{\it Structure of incidence algebras and their automorphism groups},
Bull. Amer. Math. Soc., \textbf{76} (1970), 1236-1239.



\bibitem[42]{Stanley2} 
R. Stanley, 
{\it Enumerative Combinatorics}, vol. 1. With a foreword by
Gian-Carlo Rota. Cambridge Studies in Advanced Mathematics, vol. {\bf 49},
Cambridge University Press, Cambridge, 1997.



\bibitem[43]{WangXiao} 
D.-N. Wang and Z.-K. Xiao, 
{\it Lie triple derivations of incidence algebras},
Comm. Algebra,  \textbf{47} (2019), 1841-1852.



\bibitem[44]{Wang}
Y. Wang,
{\em Lie $n$-derivations of unital algebras with idempotents},
Linear Algebra Appl., \textbf{458} (2014), 512-525.



\bibitem[45]{WangWang}
Y. Wang and Y. Wang,
{\em Multiplicative Lie $n$-derivations of generalized matrix algebras},
Linear Algebra Appl., \textbf{438} (2013), 2599-2616.




\bibitem[46]{Ward} 
M. Ward, 
{\it Arithmetic functions on rings},
Ann. of Math., {\bf 38} (1937), 725-732.




\bibitem[47]{Xiao} 
Z.-K. Xiao, 
{\it Jordan derivations of incidence algebras},
Rocky Mountain J. Math., {\bf 45} (2015), 1357-1368.




\bibitem[48]{XiaoWei} 
Z.-K. Xiao and F. Wei, 
{\em Nonlinear lie-type derivations on full matrix algebras},
Monatsh. Math., {\bf 170} (2013), 77-88.




\bibitem[49]{XiaoYang} 
Z.-K.Xiao and Y.-P.Yang, 
{\em Lie $n$-derivations of incidence algebras}, 
Comm. Algebra, \textbf{48} (2020), 105-118.



\bibitem[50]{Yang2020}
Y.-P. Yang,
{\em Nonlinear derivations of incidence algebras}, Acta Math. Hungar., 
\textbf{162} (2020),  52-61.



\bibitem[51]{Yang20211}
Y.-P. Yang,
{\em Nonlinear Lie derivations of incidence algebras}, Oper. Matrices, 
\textbf{15} (2021),  275-292.



\bibitem[52]{Yang20212}
Y.-P. Yang,
{\em Nonlinear Lie derivations of incidence algebras of finite rank},
Linear Multilinear Algebra, to appear, 
\href{https://doi.org/10.1080/03081087.2019.1635979}{https://doi.org/10.1080/03081087.2019.1635979}.



\bibitem[53]{YuZhang1} 
W.-Y. Yu and J.-H. Zhang, 
{\em Nonlinear $\ast$-Lie derivations on factor von Neumann algebras}, Linear Algebra
Appl., \textbf{437} (2012), 1979-1991.



\bibitem[54]{YuZhang2} 
W.-Y. Yu and J.-H. Zhang, 
{\em Lie triple derivations of CSL algebras}, 
Internat. J. Theoret. Phys., {\bf 52},(2013), 2118-2127.



\bibitem[55]{ZhangKhrypchenko}
X. Zhang and M. Khrypchenko, 
{\em Lie derivations of incidence algebras},  
Linear Algebra Appl., \textbf{513} (2017), 69-83.




\bibitem[56]{ZhaoLi1} 
F.-F. Zhao and C.-J. Li,  
{\em Nonlinear $\ast$-Jordan triple derivations on von Neumann algebras}, 
Math. Slovaca, \textbf{68} (2018), 163-170.

\end{thebibliography}
\end{document}